\documentclass[a4paper,reqno,11pt]{amsart}

\newcommand{\ds}{\displaystyle}

\newcommand{\tensor}{\otimes}
\newcommand{\leftsub}[2]{{\vphantom{#2}}_{#1}{#2}} 
\newcommand{\xycirc}[2]{\leftsub{#1}{\circ}_{#2}}

\newcommand{\op}{\mathcal}
\newcommand{\fr}{\mathfrak}

\newcommand{\cdc}{,\dots,}
\newcommand{\tdt}{\tensor\dots\tensor}

	\newcommand{\fcops}{\mathbb{O}}

	\newcommand{\fccyc}{\mathbb{C}}
	
	\newcommand{\fcdi}{\mathbb{DO}}

	\newcommand{\fcmod}{\mathbb{M}}

	\newcommand{\meta}{k[C_{\ast+1}]}

\usepackage{amsmath}%
\usepackage{amsthm}
\usepackage{amsfonts}%
\usepackage{amssymb}%
\usepackage{graphicx}
\usepackage{xy,amsthm,enumerate,xypic,array}  

\input{xy}
\xyoption{all}

\numberwithin{equation}{section}

\newtheorem{theorem}{Theorem}[section]
\theoremstyle{plain}

\newtheorem{assumption}[theorem]{Assumption}
\newtheorem*{nonumbertheoremA}{Theorem A}
\newtheorem*{nonumbertheoremB}{Theorem B}
\newtheorem*{nonumbertheoremC}{Theorem C}
\newtheorem*{nonumbertheoremD}{Theorem D}

\newtheorem{corollary}[theorem]{Corollary}
\newtheorem{lemma}[theorem]{Lemma}
\newtheorem{proposition}[theorem]{Proposition}

\theoremstyle{definition}
\newtheorem{definition}[theorem]{Definition}

\newtheorem{remark}[theorem]{Remark}

\newtheorem{calculation}[theorem]{Calculation}
\allowdisplaybreaks[2]

\addtocounter{MaxMatrixCols}{2}

\begin{document}

\title{Six operations formalism for generalized operads}
\author{Benjamin C. Ward}
\email{bward@math.su.se}


\begin{abstract}  This paper shows that generalizations of operads equipped with their respective bar/cobar dualities are related by a six operations formalism analogous to that of classical contexts in algebraic geometry.  As a consequence of our constructions, we prove intertwining theorems which govern derived Koszul duality of push-forwards and pull-backs.
\end{abstract}
\maketitle
\tableofcontents

\section{Introduction.}

 The crux of the homotopical algebra of dg algebras and operads is the property of cofibrancy.  To produce explicit cofibrant resolutions of these objects one may use the bar/cobar constructions, either by a single application to a Koszul dual object or by double application in general.  By the ``Feynman transform'' we refer to a generalization of the bar/cobar construction which may be applied within categories of dg operads, cyclic operads, modular operads, dioperads, properads, etc, and we will denote the Feynman transform by $D$ (a notational choice suggested by the results of this paper).  We will refer to application of the Feynman transform as (derived) Koszul duality.

The mantra of this paper holds that it is desirable to consider Koszul duality not as an aspect of these categories separately, but rather as a construction which intertwines the relationships (i.e.\ functors) between them.  The paradigm is Verdier duality.  Namely, Verdier duality provides a duality theory in the category of sheaves over a fixed space, but moreover intertwines natural operations of push-forward and pull-back associated to suitable continuous maps between spaces.  The categorical formulations of sheaf theory and of generalized operads via Feynman categories \cite{KW} tightens this analogy.  In particular, both classes of objects can be described as locally defined functors from categories whose morphisms take a prescribed simple form.  The following table presents this analogy:

\bigskip
\begin{center}
	\begin{tabular}{r||c|c}  & Verdier Duality (sheaves) &  Koszul Duality (gen.\ operads) \\ \hline\hline 
		Encoding   &    $Open(X)$,       & $\mathbb{F}$  \\ 
		category&         category of open sets              &   a Feynman category   \\  \hline
		Functors        &    presheaves                &$\mathbb{V}$-modules \\ 
		&        &                              \\ \hline
		Functors       &     sheaves             &$\mathbb{F}$-operads \\ 
		+gluing axiom      &   &                             \\ \hline
		Interposing &  $Open(Y)\to Open(X)$ &$\mathbb{F}\stackrel{f}\to\mathbb{F}^\prime$ \\  
		functor      &  via continuous $f\colon X\to Y$       &      \\ \hline
		Pull-back      &  inverse image sheaf         &  composition   \\ 
		of $f$&        $f^*$       &         $-\circ f$      \\ \hline
		Push-forward      &  direct image sheaf    & left Kan extension\\ 
		of $f$ &    $f_\ast$         &     $Lan_{f}(-)$        \\ \hline
		Dualizing  &   $D_X= RHom(-,\omega_X)$  &  $D =$ (co)bar construction or\\ 
	functor	&    & Feynman transform     \\ \hline
		Intertwining    &  push-forward w/ compact support $f_!$ & Novel construction \\
		adjunction  &   and  exceptional pull-back $f^!$   & of this paper.                    
	\end{tabular}
\end{center}
\bigskip

The existence of a pair of adjunctions intertwined by duality, along with compatible closed monoidal structures, is sometimes known as Grothendieck's six operations formalism.  In the language of symmetric monoidal categories, this structure was studied in \cite{FHM}, who disentangled the axiomatics from their implications.  Following their work, we present the ``six operations formalism'' (Definition $\ref{6def}$) as two pairs of adjunctions which satisfy the projection formula.  This structure is most useful in the presence of compatible dualizing objects, and we call such a six operations formalism ``effective'' (Definition $\ref{e6ops}$).  

To state the first result of this paper we fix a field $k$ of characteristic zero and define $\mathbb{F}$-$\op{O}ps$ to be the category of symmetric monoidal functors $\mathbb{F}\to dgVect_k$.  We recall from \cite{KW} that $\mathbb{F}$-$\op{O}ps$ is a model category with level-wise weak equivalences and fibrations.  We also recall that $\mathbb{F}$-$\op{O}ps$ is a symmetric monoidal category with level-wise tensor product.  We call the unit of this monoidal structure $I$. 

\begin{nonumbertheoremA}
	Let $f\colon \mathbb{F}_1\to\mathbb{F}_2$ be a (nice) morphism of (nice) Feynman categories.  Then:
\begin{enumerate}
	\item  The standard adjunction $f_\ast \colon \mathbb{F}_1$-$\op{O}ps \leftrightarrows \mathbb{F}_2$-$\op{O}ps \colon f^\ast$ extends to a six operations formalism from $\mathbb{F}_1\text{-}\op{O}ps$ to $\mathbb{F}_2\text{-}\op{O}ps$.
	\item  The six-operations formalism of (1) passes to the homotopy categories.  Moreover, it extends to an effective six-operations formalism from $Ho(\mathbb{F}_1\text{-}\op{O}ps)$ to $Ho(\mathbb{F}_2\text{-}\op{O}ps)$ with dualizing object $D(I_2)$.
\end{enumerate}	
	
\end{nonumbertheoremA}

The appearance of the word ``nice'' in the theorem indicates, of course, that some conditions are required for the theorem to hold.  For the first statement, these conditions are spelled out in Definition $\ref{admisdef}$ and are comprised of two technical requirements about factorization of morphisms in the image of $f$, along with the requirement that $f$ is faithful as a functor.  For the second statement, we further restrict our attention to a context in-which $D$ is suitably defined.  This context, given in \cite{KW}, is provided by so-called cubical Feynman categories, a notion we recall in Appendix $\ref{FCA}$.  While it may be possible to weaken these requirements (discussed in section $\ref{discussion}$ below), we emphasize that the theorem applies to the expected examples as we explain below in Section $\ref{examples}$.

Let us emphasize right from the beginning that the proof of Theorem A is not simply an exercise in existence of adjoints in symmetric monoidal categories.  Indeed, our construction of the push-forward $f_!$ depends crucially on the axiomatics of a Feynman category and uses the groupoid of vertices $\mathbb{V}\subset\mathbb{F}$ in an essential way.  This should not be surprising in light of statement (2) of the theorem which indicates compatibility between our constructions and the Feynman transform.  However, after constructing $f_!$ we can reduce the existence of its left adjoint $f^!$ to an ambidexterity result at the level of the groupoid of vertices (Lemma $\ref{amblem}$).  After constructing the adjunction $(f^!,f_!)$, the main technical requirement is then to verify the projection formula (Proposition $\ref{projformula}$).  

As an immediate consequence to Theorem A we find intertwining theorems at the level of homotopy categories which state that $f^!D\sim Df^\ast$ and $f_\ast D \sim Df_!$.  As a consequence of the explicit nature of our constructions of $f^!$ and $f_!$, we are able to exhibit maps realizing these weak equivalences and to moreover show that they are actually isomorphisms:

\begin{nonumbertheoremB}  Let $f\colon \mathbb{F}_1\to\mathbb{F}_2$ be as in Theorem A.  Then there are isomorphisms of functors:
	\begin{equation*}
f^!D\cong Df^\ast \ \ \ \text{and} \ \ \  f_\ast D\cong Df_!
	\end{equation*}
\end{nonumbertheoremB}

This theorem allows us to interpret our constructions as liftings of the intertwined push-forward and pull-back from the homotopy categories to the naive categories:

\begin{corollary}  There exist weak equivalences:
	\begin{equation*}
	Df^!D\stackrel{\sim}\to f^\ast \ \ \ \text{and} \ \ \ Df_\ast D\stackrel{\sim}\to f_! 
	\end{equation*}
\end{corollary}

In general, the four functors which comprise the adjunctions $(f_\ast,f^\ast)$ and $(f^!,f_!)$ are distinct.  In many cases, however, these functors satisfy some ambidexterity which (in Section $\ref{WGsec}$) we characterize via:

\begin{nonumbertheoremC} (Ambidexterity Results)
	\begin{enumerate}
		
	\item 	If for every morphism $\psi\in Mor(\mathbb{F}_2)$ of the form $\psi \colon f(X)\to v$
	there exists $\sigma \in Aut(v)$ and  $\phi\colon X\to w \in Mor(\mathbb{F}_1)$ with $\psi=\sigma f(\phi)$, then $(f^!,f_\ast,f^\ast)$ is a triple of adjoint functors.

		
		\item  If for every morphism $\psi\in Mor(\mathbb{F}_2)$ of the form $\psi\colon Y \to f(w)$, there exists $\sigma\in Aut(Y)$ and $\phi\colon X\to w \in Mor(\mathbb{F}_1)$ with $\psi=f(\phi)\sigma$, then $(f_\ast,f^\ast,f_!)$ is a triple of adjoint functors.
	\end{enumerate}

\end{nonumbertheoremC} 

In the parlance of \cite{FHM} the first criterion then gives a necessary condition for the ``Grothendieck context'' whereas the second criterion gives a necessary condition for the ``Wirthm{\"u}ller context''.  Interpreting Theorem C in the examples of Feynman categories built from graphs, we develop the following heuristic:  inclusions often realize the Wirthm{\"u}ller context, symmetrizations often realize the Grothendieck context.  We give specific examples of these heuristics in Section $\ref{examples}$.

The principal notion of duality on the categories $\mathbb{F}$-$\op{O}ps$ that we consider is bar-cobar duality.  However, if we restrict our attention to sub-categories of quadratic objects, there is another candidate namely quadratic duality, denoted $(-)^!$.  There is an analog of Theorem A for quadratic subcategories, where $\tensor$-product is replaced by (a generalization of) Manin products, see Section $\ref{sec:quad}$.  This in particular allows us to characterize when $f_!$ and $f^!$ can be described by intertwining with $(-)^!$.  But we emphasize that this is not generally the case, and the quadratic story is a specialization of the general intertwining story summarized above.

Finally, note that the intertwining theorems above rely on the existence of dualizing objects which we created via the Feynman transform.  On the other hand, our analogy with Verdier duality suggests a way to create a dualizing object $\omega_\mathbb{F}$ by pulling back via $p^!$, where $p$ is the map to the analog of a point.  In Section $\ref{vdkd}$ we show that these two ways of producing dualizing objects coincide.  In particular, the analog of the Verdier dual is the Feynman transform:

\begin{nonumbertheoremD}  $D(-)\cong cohom(\omega_\mathbb{F},-)$ in $Ho(\mathbb{F}$-$ops)$.
\end{nonumbertheoremD} 

Morphisms from the Feynman transform correspond to Maurer-Cartan elements in an associated $L_\infty$ algebra and the associated deformation complexes may be studied via the following corollary (Corollary $\ref{prescor}$) of Theorem D:

\begin{corollary}  Given $f$ as in Theorem A, $f^!$ preserves dualizing objects and $f_!$ preserves deformation complexes.
\end{corollary}

\subsection{Applications: an overview}  Let us now discuss a few applications of our results for particular morphisms of Feynman categories.  Further results and discussion can be found in Section $\ref{examples}$.

{\bf The derived modular envelope.} There is a morphism between Feynman categories for cyclic and modular operads given by inclusion.  In this case we find the ``Wirthm{\"u}ller context'' and hence the triple of adjoint functors $(f_\ast,f^\ast,f_!)$.  Our results prove:
\begin{corollary}
	Let $\op{O}$ be a cyclic operad, let $f_\ast$ be the modular envelope construction, and let $f_!$ be extension by zero.  Then there is an isomorphism of ($\fr{K}$-twisted) modular operads:
	\begin{equation*}
f_\ast (D(\op{O}))\cong D(f_!(\op{O}))
	\end{equation*}

\end{corollary}

This result is most interesting in the case that $\op{O}$ is a Koszul cyclic operad, in that it gives us a model for the derived modular envelope:

\begin{corollary}\label{gccor}
	Let $\op{O}$ be a Koszul cyclic operad and let $\mathbb{L}f_\ast$ denote the derived modular envelope.  Then  
		\begin{equation*}
		\mathbb{L}f_\ast(\op{O}^!)\sim D(f_!(\op{O}))
		\end{equation*}
\end{corollary}
This result is particularly interesting in view of Kontsevich's theorem \cite{Kont2, CV} and generalizations, \cite{CKV},  which relate graph complexes associated to cyclic operads -- via the functor $Df_!$ -- to homology of (generalizations of) moduli spaces of Riemann surfaces and outer automorphisms of free groups.  Corollary $\ref{gccor}$ allows us to alternatively view graph homology as measuring the failure of $f_\ast$ to preserve the Koszul weak equivalence $D\op{O}\stackrel{\sim}\to \op{O}^!$.  As a specific application, we consider the morphism $f_\ast (D\op{L}ie) \to f_\ast(\op{C}om)$, and using known information about the Lie graph complex and the homotopy involutivity of the Feynman transform, we compute (Calculation $\ref{calculation})$:

\begin{corollary} $\chi(Df_\ast\op{C}om(1,n))=(-1)^n(n-1)!/2$.
\end{corollary}

The input for this calculation was the calculation $\chi(f_\ast D\op{L}ie(1,n))=2^{n-2}$ of \cite{CHKV}, and it agrees with the results of \cite{CGP} as we discuss below.

{\bf Generalized Drinfeld double.} 
We consider now the morphism $\phi$ between the Feynman category for dioperads and the Feynman category for cyclic operads which forgets directions on the edges of graphs.

For a dg vector space of finite type $A$, let $A^\ast$ denote its linear dual and define $\mathsf{d}(A):= A\oplus A^\ast$.  We equip $\mathsf{d}(A)$ with the non-degenerate bilinear form $\langle a\oplus \eta, b\oplus \xi \rangle= \eta(b)+ \xi(a)$.  We first must establish the ``binomial theorem'' for cyclic operads, which states that $\phi_!(End^{di}_A)= End^{cyc}_{\mathsf{d}(A)}$.  We then have the following immediate consequence:

\begin{corollary}  Let $\op{O}$ be a cyclic operad.  A $\phi^!(\op{O})$-algebra structure on $A$ is equivalent to an $\op{O}$-algebra structure on $(\mathsf{d}(A), \langle-,-\rangle)$.    
\end{corollary}

Considering this result in the case $\op{O}=\op{L}ie$, the cyclic operad encoding Lie algebras, we compute that $\phi^!(\op{L}ie)$ is the dioperad encoding Lie bialgebras.  Hence we see the adjunction $(\phi^!,\phi_!)$ as a generalization of the correspondence between Lie bialgebras and so-called Manin triples established in \cite{Drinfeld}. 

{\bf  Moduli space actions on deformation complexes}.  Consider the notion of a planar dioperad.  This is a collection of objects $\op{O}(n,m)$ having operations parametrized by planar directed graphs of genus zero.  The dualizing complex in the category of planar dioperads encodes the notion of a $V_\infty$ algebra due to \cite{TZ}.  

Recall that the shifted homology of the moduli space of punctured Riemann spheres forms an operad called the gravity operad which has a Koszul resolution by the dg operad $\op{G}rav_\infty$.  Using our result on the preservation of deformation complexes we prove:

\begin{corollary}  The Lie bracket on the deformation complex of a $V_\infty$-algebra extends to an action of a $\op{G}rav_\infty$-algebra.
\end{corollary}

More generally we emphasize the paradigm that the calculation of the dg operad of natural operations on deformation complexes (in the sense of 6.6.4 of \cite{KW}) is preserved by forgetting directions of graphs.

There are more applications to be found by applying our general results to specific morphisms.  See Section $\ref{examples}$.

\subsection{Discussion and future directions.}  \label{discussion}

The {\it raison d'{\^e}tre} of Feynman categories \cite{KW} was that it would allow the study of generalizations of operads via the rather well understood language of symmetric monoidal categories.  We view the results of this paper as a realization of this assertion.  In particular we had easy access to the formulation of the six operations formalism via symmetric monoidal categories given in \cite{FHM}.   Of course it is not difficult to translate our results into other languages, such as multi-categories or algebras over colored operads.  Appendix $\ref{FCA}$ contains a reference guide for Feynman categories and an overview of our conventions regarding signs, parity and the Feynman transform.

A parallel between Verdier duality and Koszul duality for operads was there from the very beginning; see \cite{GiK} where the authors give an interpretation of linear operads as (nice) sheaves on a particular space, such that Verdier duality coincides with their dual dg operad $D$, from which we borrow notation.  Taking this work a step further, in \cite{LVor}  the authors construct a space of metric graphs on which cyclic operads determine sheaves.  They show that the Verdier dual of the sheaf associated to $\op{O}$ is the sheaf associated to $D\op{O}$ and study ensuing connections to graph complexes.  Our work suggests a possible generalization of these works along the following lines: functorially associate a space to a cubical Feynman category $\mathbb{F}$ such that functors correspond to sheaves, and show that the six operations formalism, in the classical sense, coincides with that in Theorem A.  This would give an interesting, if somewhat round-about, interpretation of the constructions presented here.

There are several directions in which the results of this paper may be generalized.  One could consider more general base categories; interesting examples include spectra or $k[G]$-modules.  It is worth pointing out, however, that these generalizations will be of a somewhat different character since it may no longer be the case that Theorem B holds on the nose (see \cite{Wardsdp} for an example of this).   One could also look to loosen the conditions imposed on the encoding categories, or the morphisms between them.   For example, it may be possible to loosen the faithfulness requirement of Definition $\ref{admisdef}$ in some contexts by a symmetrically invariant averaging over all pull backs.

Finally, we remark that there is a valuable perspective of why this structure should exist which is internal to the language of operads.  The cubical Feynman categories have underlying colored operads which are quadratic and self-dual (up to questions of parity).  Thus, passing to an enriched context, a morphism $f\colon\mathbb{F}_1\to\mathbb{F}_2$, gives rise to a Koszul dual morphism $g\colon\mathbb{F}_2\to\mathbb{F}_1$, and one may compare $g^\ast$ and $f_!$.  From this perspective it is perhaps surprising to note that our conditions for the existence of $f_!$ (Definition $\ref{admisdef}$) do not imply self-duality of the enriched colored quadratic operad.  If one were looking to generalize the conditions under which the six operations are expected, keeping this perspective in mind seems particularly valuable.

\subsection{Notation}  In the body of the text we will not continue to use the standard sheaf theoretic notation for the four functors $f^\ast,f_\ast,f^!,f_!$.  Our notation and the translation is given as:

\begin{center}
	\begin{tabular}{c||c|c|c}
		
		Sheaves    & $\mathbb{F}$-$\op{O}ps$     & Derived  & Quadratic  \\ \hline
		$f^\ast$    & $R$  & $R$   & $R$ \\ \hline
		$f_\ast$    & $L$  & $LD^2$   & $L$ \\ \hline
		$f^!$      & $R^!$   & $DRD=R^!D^2$   & $R(-^!)^!$ \\ \hline
		$f_!$       & $L^!$  & $DLD=L^!$   & $L(-^!)^!$ \\ 
	\end{tabular}

\end{center}

Other notation: $k$ is a field of characteristic $0$,  $dgVect_k$ is the category of differential graded (dg) $k$-vector spaces, $s(-)$ and $t(-)$ associate the source and target object to a morphism.

\section{Categorical six operations formalism.} 

In this section we give a definition of six operations formalisms at the level of symmetric monoidal categories.  The material in this section mainly follows \cite{FHM}, but has been adapted a bit to the situation that we shall consider in future sections.  In particular, we will consider categories of covariant functors (as opposed to sheaves), and so {\it all of the natural transformations go in the opposite direction as would be expected} (from the latter) and in particular the handedness of all adjunctions is the opposite of \cite{FHM}.  We formalize this dichotomy by introducing a handedness (left vs. right) in the definition.

We work in the context of co-closed symmetric monoidal categories.  To say a symmetric monoidal category $(\op{C},\tensor, I)$ is co-closed means that there is a bifunctor denoted $cohom\colon \op{C}\times\op{C}^{op}\to\op{C}$ such that for each object $a\in\op{C}$ there is an adjunction $(cohom(-,a), a\tensor -)$.

\begin{definition} \label{6def} Let $\op{B}$ and $\op{C}$ be co-closed symmetric monoidal categories.  A (right) six-operations formalism from $\op{B}$ to $\op{C}$ is the following data:

\begin{enumerate}[i.]
\item  An adjunction $f_\ast\colon \op{B}\rightleftarrows\op{C}: f^\ast$.
\item  An adjunction $f^!\colon \op{C}\rightleftarrows\op{B}: f_!$.

	\item  A natural transformation $\pi \colon \tensor \circ (id\times f_!)\Rightarrow f_!\circ\tensor\circ(f^\ast\times id)$ of functors (with source and target $\op{C}\times \op{B} \rightrightarrows \op{C}$).
\end{enumerate}	

\begin{flushleft}
		subject to the following conditions:	
\end{flushleft}

\begin{enumerate}		
	\item  $f^\ast$ is strong symmetric monoidal.
	\item  For each pair of objects $c$ and $d$, $\pi_{c,d}\colon c\tensor f_!(d) \stackrel{\cong}\to f_!(f^\ast(c)\tensor d)$ is an isomorphism. 
\end{enumerate}	
\end{definition}

\begin{remark}  It may be helpful to keep the following points in mind:
\begin{itemize}
	\item  The notation $f^\ast$ is suggestive of an ``underlying'' $f$ which will exist in our examples, but need not exist at this level of generality.  The terminology ``from $\op{B}$ to $\op{C}$'' is suggestive of this morphism.
	\item  The functors $f_\ast, f^!, f_!$ are not assumed to be symmetric monoidal.  It is a consequence of the definition that $f_\ast$ is op-lax monoidal.
	\item  The ``six'' operations referred to in the terminology are $(f_\ast,f^\ast,f^!,f_!,cohom,\tensor)$, even though (of course) each category has its own $(cohom, \tensor)$ adjunction.
	\item The maps $\pi_{c,d}$ will be called the projection formula maps and the condition (2) above will be called the projection formula.
	\item The ``right'' in the definition refers to the handedness of the monoidal product.  We suggest the terminology ``left'' for the mirror image structure having the handedness of all three adjunctions reversed.  This would encompass the examples appearing in sheaf theory and in the treatment \cite{FHM}.
\end{itemize}	
\end{remark}

Before examining consequences of this structure we need the notion of a dualizing object.  The units of the adjunctions $(cohom(-,a), a\tensor -)$ give rise to co-evaluation maps $b \to  a\tensor cohom(b,a) \cong cohom(b,a)\tensor a$.  Taking the adjoint we have maps $\eta_a^b\colon cohom(b,cohom(b,a))\to a$.

\begin{definition}  An object $b$ in a co-closed symmetric monoidal category $\op{C}$ is {\it dualizing} if $\eta_a^b$ is an isomorphism for every object $a\in\op{C}$.  
\end{definition}

Give an object $b$ we define the contravariant functor $D_b(-):= cohom(b,-)$.

\begin{lemma}  The six operations formalism has the following consequences:
	\begin{enumerate}

\item There are isomorphisms of functors:  
\begin{equation*}
D_{f^!(\omega)}f^\ast\cong f^! D_\omega \ \ \ \text{ and } \ \ \ f_\ast D_{f^!(\omega)} \cong D_\omega f_!
\end{equation*}

\item If $\omega\in\op{C}$ is a dualizing object, there are isomorphisms of functors: 
\begin{equation*}
D_{f^!(\omega)}f^\ast D_\omega\cong f^! \ \ \ \text{ and } \ \ \ D_\omega f_\ast D_{f^!(\omega)} \cong f_!
\end{equation*}

\item If $f^!(\omega)\in\op{B}$ is a dualizing object, there are isomorphisms of functors:
\begin{equation*}
f^\ast\cong D_{f^!(\omega)}f^! D_\omega \ \ \ \text{ and } \ \ \ f_\ast  \cong D_\omega f_!D_{f^!(\omega)}
\end{equation*}

\end{enumerate}

\end{lemma}
\begin{proof}  Using adjointness and the projection formula we see that, for every pair of objects $c_1,c_2\in ob(\op{C})$, there are natural isomorphism of functors $\op{B}\to Set$ of the form:
	\begin{eqnarray}\label{Yoneda}
	 Hom_\op{B}(cohom(f^!(c_1),f^\ast(c_2)), -) & \cong &  Hom_\op{C}(c_1, f_!(f^\ast(c_2)\tensor -)) \\   & \cong & Hom_\op{C}(c_1,c_2\tensor f_!(-)) \nonumber  \\ & \cong & Hom_\op{B}(f^!(cohom(c_1,c_2)),-).	 \nonumber
	\end{eqnarray}

Statement (1) then follows from the Yoneda lemma (evaluating at $\omega=c_1$).  Statements (2) and (3) follow from the definition of dualizing. \end{proof}

Let us emphasize that the structure of a six operations formalism is not inherently interesting, e.g.\ we might take all four functors to be zero.  But the preceding lemma shows that if one is interested in such a dualizing functor, it can be studied in particular cases via the six operations, due to statement (1).  On the other hand, if one is interested in the six operations to begin with, finding dualizing objects will permit the study of these functors via statements (2) and (3).  In particular, this structure is most useful in the presence of such dualizing objects, hence we make the following definition:

\begin{definition}\label{e6ops}
	An effective six-operations formalism from $\op{B}$ to $\op{C}$ is a six-operations formalism from $\op{B}$ to $\op{C}$ and a dualizing object $\omega\in\op{C}$ such that $f^!(\omega)$ is dualizing.
\end{definition}

\section{Quillen adjunctions from morphisms of Feynman categories.}  

The six operations formalisms that we consider in this paper will be associated to morphisms of Feynman categories.  We remind the reader that additional prerequisites, references, and technical assumptions related to Feynman categories appear in Appendix $\ref{FCA}$. 

Fix a morphism of Feynman categories $\phi\colon\mathbb{F}_1\to\mathbb{F}_2$ and a cocomplete and closed symmetric monoidal category $\op{C}$.  Associated to $\phi$ there is an adjunction:
\begin{equation*}
L_\phi \colon {\mathbb{F}_1}\text{-}\op{O}ps_\op{C} \leftrightarrows {\mathbb{F}_2}\text{-}\op{O}ps_\op{C} \colon R_\phi
\end{equation*}
in which $R_\phi$ is given by composition, $R_\phi(\op{O})=\op{O}\circ \phi$, and $L_\phi$ is given by left Kan extension, $L_\phi(\op{P})=Lan_\phi(\op{P})$.  We often suppress the subscript if $\phi$ is fixed.  This adjunction gives rise to the following not necessarily commutative square of adjunctions.  We emphasize to the reader that these functors and their accompanying notation will be used frequently throughout.
\begin{equation*}
\xymatrix{ {\mathbb{V}_1}\text{-}Mods_\op{C} \ar[rr]^{F_1} \ar[dd]^{l} & & \ar@<1ex>[ll]^{G_1} {\mathbb{F}_1}\text{-}\op{O}ps_\op{C} \ar[dd]^{L} \\ & & \\ \ar@<1ex>[uu]^{r}  {\mathbb{V}_2}\text{-}Mods_\op{C} \ar[rr]^{F_2} & &\ar@<1ex>[ll]^{G_2} \ar@<1ex>[uu]^{R} {\mathbb{F}_2}\text{-}\op{O}ps_\op{C}}
\end{equation*}

Here $(F,G)$ and $(l,r)$ are adjunctions which arise for the particular morphisms of inclusion(s) $\iota\colon\mathbb{V}\hookrightarrow \mathbb{F}$ and restriction $\phi\colon\mathbb{V}_1\to\mathbb{V}_2$.

\begin{lemma}\label{ntlem}  With respect to the diagram:
\begin{enumerate}
\item  $r G_2= G_1R$
\item  $F_2l \cong L F_1$

and consequently there exist natural transformations:

\item  $F_1r \Rightarrow R F_2$
\item  $l G_1 \Rightarrow G_2 L$ 
\end{enumerate}
\end{lemma}
\begin{proof}{\bf The first statement:}  By the definition of a morphism of Feynman categories, $\iota_2\phi=\phi\iota_1$.  Thus if $\op{O}\in{\mathbb{F}_2}\text{-}\op{O}ps_\op{C}$ and $v\in\text{ob}(\mathbb{V}_1)$, then:
\begin{equation*}
r (G_2(\op{O}))(v) :=  G_2(\op{O})(\phi(v))=\op{O}(\iota_2\circ \phi(v)) = \op{O}(\phi\circ \iota_1(v)) =: G_1(R(\op{O}))(v)
\end{equation*}

{\bf The second statement:} follows from the first statement and the Yoneda lemma.

{\bf The third statement:}  $F_1r \Rightarrow F_1r G_2F_2 = F_1G_1R F_2 \Rightarrow R F_2$.

{\bf The fourth statement:} $l G_1 \Rightarrow G_2F_2 l G_1 \cong G_2 L F_1G_1 \Rightarrow G_2 L$. \end{proof}

\begin{lemma}  $L,R,l,r$ are preserved by composition: $L_{\phi_1\phi_2}=L_{\phi_1}L_{\phi_2}$, $R_{\phi_1\phi_2}=R_{\phi_2}R_{\phi_1}$ etc.
\end{lemma}

We consider $\mathbb{F}$-$\op{O}ps_\op{C}$ to be a symmetric monoidal category via the $\tensor$-product in $\op{C}$; i.e.\ for $X\in ob(\mathbb{F}),  (\op{O}\tensor\op{P})(X):=\op{O}(X)\tensor\op{P}(X)$, and likewise for morphisms.  Then:

\begin{lemma} $R$ is a symmetric monoidal functor.  In particular $R$ preserves the $\tensor$-unit.
\end{lemma}

\begin{remark}\label{Quillenadj}
	If $\op{C}=dgVect_k$,  then $\mathbb{F}$-$\op{O}ps$ is a model category whose quasi-isomorphisms are level-wise weak equivalences (\cite{KW} Theorem 8.2.1).  We will often be concerned with the case in which $\phi$ is faithful, which implies that $(L,R)$ is a Quillen adjunction (\cite{KW} Lemma 8.28).
\end{remark}

\begin{remark}\label{LRminusrmk}  If $\phi$ is cubical (see Appendix $\ref{FCA}$), then $L$ and $R$ preserve skew invariance and thus restrict to an adjunction:
	\begin{equation*}
	L^- \colon {\mathbb{F}^-_1}\text{-}\op{O}ps_\op{C} \leftrightarrows {\mathbb{F}^-_2}\text{-}\op{O}ps_\op{C} \colon R^-
	\end{equation*}
\end{remark}

\subsection{Ambidexterity}
We now impose $\op{C}=dgVect_k$ and establish the following ambidexterity result.

\begin{lemma}\label{amblem}
	The functors $(r,l)$ form an adjoint pair.  Explicitly, $r$ is both a left and right adjoint of $l$ and vice-versa.
\end{lemma}
\begin{proof} First, we recall that $r$ is given by restriction of $\phi$ to the groupoids of vertices $\phi\colon\mathbb{V}_1\to\mathbb{V}_2$ and $l$ is given by left Kan extension of this restriction.  In particular, using our skeletal assumption (Appendix $\ref{FCA}$) we may avoid colimits and just write:
	\begin{equation*}
	lA(v):= \bigoplus_{\substack{w\in\mathbb{V}_1 \\ \phi(w) = v}} \left[A(w)\tensor k[Aut(v)]\right]_{Aut(w)}
	\end{equation*}	
	
We know that  $(l,r)$ form an adjoint pair, and we endeavor to show that $(r,l)$ do as well.  For this we will construct a natural isomorphism:
	\begin{equation*}
	Hom_{\mathbb{V}_1\text{-}Mods}(rE,A)\cong Hom_{\mathbb{V}_2\text{-}Mods}(E,lA)
	\end{equation*}
	To begin, suppose we are given $\lambda\colon rE\to A$ and we will construct $\psi\colon E\to lA$.  We thus start with an $Aut(w)$-equivariant map $\lambda_w\colon E(\phi(w))\to A(w)$ for each $w\in \mathbb{V}_1$.  We will define an $Aut(v)$-equivariant map $\psi_v\colon E(v)\to lA(v)$ as the direct sum of morphisms indexed over the objects of $\mathbb{V}_1$:
	\begin{equation*}
	\psi_w\colon E(\phi(w))\to \left[A(w)\tensor k[Aut(\phi(w))]\right]_{Aut(w)}
	\end{equation*}	
	Explicitly: $\psi_v=\oplus_{w\mapsto v} \psi_w$.  Define $\psi_w$ by:
	\begin{equation*}
	\psi_w(x) = \frac{1}{|Aut(w)|}\ds\sum_{\sigma \in Aut(v)}[\lambda_{w}(\sigma x)\tensor \sigma^{-1}]
	\end{equation*}
 Notice that each $\psi_w$ is $Aut(v)$-equivariant, and thus $\psi_v$ is.  Here we are using that $Aut(v)$ acts on the right factor only (on the left hand side of the right factor).  It may be the case that there does not exist a $w$ with $\phi(w)=v$, in which case the definition results in an empty sum and hence $\psi_v=0$.  Whence the map:  $Hom_{\mathbb{V}_1\text{-}Mods}(rE,A)\to Hom_{\mathbb{V}_2\text{-}Mods}(E,lA)$.
	
	An inverse to this map is given as follows.  Suppose we are given $\psi\colon E\to lA$.  For $w\in\mathbb{V}_1$ restricting to the $w$-summand (as above) we have a map:
	\begin{equation*}
	\psi_w\colon E(\phi(w))\to \left[A(w)\tensor k[Aut(\phi(w))]\right]_{Aut(w)}
	\end{equation*}	
	Writing $v:=\phi(w)$, this map is $Aut(v)$-equivariant (since $Aut(v)$ acts on each summand of $lA(v)$ individually).  For $x\in E(v)$ we may write: $\psi_w(x)=[\sum_{g\in Aut(v)}y_g\tensor g]$, and we define $ \lambda_w :E(\phi(w))\to A(w)$ by:
	\begin{equation*}
	\lambda_w(x) = \ds\sum_{h\in Aut(w)} h\cdot y_{\phi(h)} 
	\end{equation*}
	
	Having given the construction, the remaining details of the proof are straightforward.  One can easily check that $\lambda$ is well defined and $Aut(w)$-equivariant, and then that the assignments $\lambda\leftrightarrow \psi$ are inverses of each other. \end{proof}

One way to view the results of this paper is that we measure the failure of ambidexterity (Lemma $\ref{amblem}$) for Feynman categories which are not groupoids.

\section{The functor $L^!$.}\label{Lsec}  We continue with $\phi\colon\mathbb{F}_1\to\mathbb{F}_2$ as above.  The purpose of this section is to give the construction of the exceptional push-forward 
\begin{equation*}
L^!_\phi\colon\mathbb{F}_1\text{-}\op{O}ps\to \mathbb{F}_2\text{-}\op{O}ps
\end{equation*}
which we denote by $L^!$ if $\phi$ is fixed.  In this section we also establish the intertwining theorem $DL^!\cong LD$ and the projection formula. 

\subsection{Criterion for the existence of $L^!$}
The functor $L^!$ exists quite generally but not universally, so in the interest of restricting to when it exists we make the definition:

\begin{definition}\label{admisdef}  The functor $\phi$ is called admissible if it is faithful and satisfies the following two ``factorization axioms''.  First, if $\phi(g)=f\circ \phi(l)$ in $Mor(\mathbb{F}_2)$ then $f\in im(\phi)$.  Second, if $\phi(g)=f\circ h$ in $Mor(\mathbb{F}_2)$ then there exist a decomposition $g=g_1\circ g_2$ and an automorphism $\sigma$ such that $f=\phi(g_1)\sigma^{-1}$ and $h=\sigma\phi(g_2)$.

\end{definition}

We will see in Section $\ref{examples}$ that the examples of interest are in fact admissible.  Let us briefly give an informal description of why this is so.  These examples relate Feynman categories whose morphisms $f$ have underlying graphs $\Gamma_f$.  These graphs may carry decorations and we may also impose restrictions upon which types of graphs we choose to consider.

A factorization of a morphism $f$ as $f=h\circ g$ specifies a nesting of $\Gamma_f$ comprised of an ``external'' graph $\Gamma_h$ whose vertices are blown-up and labeled with ``internal'' graphs corresponding to the components of $\Gamma_g$.  The factorization axioms then assert that such a nested graph appears in the image of $\phi$ if and only if both its internal nested components and its external graph appear in the image.  This will be the case for morphisms given by inclusions of graphs, as well as for morphisms which forget or alter decorations, provided those decorations are determined locally, i.e.\ at vertices.

To see that functors relating these graphical examples are faithful requires the fact that a morphism is not specified just by its underlying graph, but carries the data of how the graph was assembled.  So to be faithful says that if we specify pre-images of the vertices of a graph and we specify a method of assembly of our graph, then there is at most one assembly of the pre-images whose image under $\phi$ is the graph in question.  This is clearly the case for inclusions of one class of graphs into another, but it is also the case for forgetting decorations such as directed structures or colors, since specifying a pre-image of the vertices remembers these decorations.  Examples are discussed in Section $\ref{examples}$.

\begin{assumption} From now on we will assume $\phi$ is admissible unless stated otherwise.	
\end{assumption}

\subsection{Definition of $L^!$}  Let $\op{P}\in \mathbb{F}_1\text{-}\op{O}ps$.  We first define the $\mathbb{V}_2$-module underlying $L^!(\op{P})$: for an object $v\in\mathbb{V}_2$ we define,
\begin{equation}\label{ldualeq}
L^!(\op{P})(v):= l\circ G_1(\op{P})(v) = \bigoplus_{\substack{w\in\mathbb{V}_1 \\ \phi(w) = v}} \left[\op{P}(w)\tensor k[Aut(v)]\right]_{Aut(w)}
\end{equation}
where the $Aut(w)$ subscript indicates the balanced $\tensor$-product, explicitly: \begin{equation*}
[p\tensor \sigma]= [\op{P}(\lambda)(p)\tensor\sigma\phi(\lambda^{-1})] \ \ \ \ \ \ \ \ \text{ for } \ \ \ \ \ \ \ \ \lambda \in Aut(w), \sigma \in Aut(v), \text{ and } p\in\op{P}(w).
\end{equation*}
and we often abuse notation by writing $\op{P}(w)\tensor_w k[Aut(v)] := \left[\op{P}(w)\tensor k[Aut(v)]\right]_{Aut(w)}$.  Note the sum in Equation $\ref{ldualeq}$ could be empty; an empty sum is zero by definition.

On morphisms in $\mathbb{V}_2$, i.e. for an automorphism $\tau\colon v\to v$ we define
\begin{equation*}
L^!(\tau)([p\tensor\sigma])=[p\tensor\tau\sigma]
\end{equation*}
and extend linearly.  Note that this definition is independent of choice of representative.  Whence $L^!(\op{P})$ as a $\mathbb{V}_2$-module.

We then extend $L^!(\op{P})$ strict monoidally to all objects of $\mathbb{F}_2$ and it suffices to define the image of a generating morphism $f \colon X\to v_0$.  Let's say $X=v_1\tdt v_n$.  We must then define:
\begin{equation*}
\bigotimes_{i=1}^n\left[\bigoplus_{\substack{w\in\mathbb{V}_1 \\ \phi(w) = v_i}} \left[\op{P}(w)\tensor k[Aut(v_i)]\right]_{Aut(w)}\right] \stackrel{L^!(\op{P})(f)}\longrightarrow \bigoplus_{\substack{w\in\mathbb{V}_1 \\ \phi(w) = v_0}} \left[\op{P}(w)\tensor k[Aut(v_0)]\right]_{Aut(w)}
\end{equation*}
To do this we define the image on summands and extend linearly.  To this end choose $w_i$ so that $\phi(w_i)=v_i$ for $i=0\cdc n$; if ever there does not exist such a $w_i$ then the source (if $i\neq 0$) or the target (if $i=0$) of $L^!(\op{P})(f)$ is $0$, and so the morphism is then defined to be $0$ on these summands.  We must then define:
\begin{equation*}
\bigotimes_{i=1}^n \op{P}(w_i)\tensor_{w_i} k[Aut(v_i)] \stackrel{L^!(\op{P})(f)}\longrightarrow \op{P}(w_0)\tensor_{w_0} k[Aut(v_0)]
\end{equation*}

To do this we first ask; does there exist $g\colon \tensor_{i=1}^nw_i\to w_0$ and $\sigma\in Aut(v_0)$ such that $f=\sigma\circ\phi(g)$?  If the answer is no we define $L^!(\op{P})(f)$ to be $0$ on this summand.  If the answer is yes, we make a choice of such $g$ and $\sigma$ and define:
\begin{equation*}
L^!(\op{P})(f)(\tensor_i [(p_i\tensor id_{v_i})]):= [\op{P}(g)(\tensor_i p_i)\tensor \sigma]
\end{equation*}
for $p_i\in \op{P}(w_i)$.  We do this for all generating morphisms in $\mathbb{F}_2$.  We then define:
\begin{equation*}
L^!(\op{P})(f)(\tensor_i [p_i\tensor \sigma_i]):=L^!(\op{P})(f\circ \tensor_i\sigma_i)(\tensor_i [p_i\tensor id_{v_i}]) 
\end{equation*}

Finally, we extend $L^!(\op{P})$ strict monoidally over products of generating morphisms.  

\begin{lemma}\label{wdlem}  As presented above, $L^!(\op{P})$ is a well defined symmetric monoidal functor.
\end{lemma}

\begin{proof}
Note that if $L^!(\op{P})$ is a well defined functor then it will be symmetric monoidal by construction.  To show it is well defined we check that this definition is independent of the choices made.  First we check that the definition is independent of any choice of representatives of coinvariants.  This means checking that:
\begin{equation}\label{eq123}
L^!(\op{P})(f)(\tensor_i[p_i\tensor \sigma_i])=L^!(\op{P})(f)(\tensor_i[\op{P}(\lambda_i^{-1})(p_i)\tensor \sigma_i\phi(\lambda_i)])
\end{equation}

Observe that there exists $g\colon \tensor_{i=1}^n w_i \to w_0$ and $\sigma\in Aut(v_0)$ such that $\sigma\phi(g)= f\circ\tensor_i\sigma_i\phi(\lambda_i)$ if and only if there exists $g^\prime\colon\tensor_{i=1}^n w_i \to w_0$ and $\sigma^\prime\in Aut(v_0)$ with $\sigma^\prime\phi(g^\prime)= f\circ\tensor_i\sigma_i$.  If there does not exist such $g,\sigma$ and $g^\prime,\sigma^\prime$ then equation $\ref{eq123}$ holds since both sides are zero.

So let's assume there exist such $g,\sigma$ and $g^\prime,\sigma^\prime$, which are thus related by $g=g^\prime\circ \tensor_i \lambda_i$ and $\sigma=\sigma^\prime$.  We have:
\begin{align*}
L^!(\op{P})(f)(\tensor_i[\op{P}(\lambda_i^{-1})(p_i)\tensor\sigma_i \phi(\lambda_i)]) 
& = L^!(\op{P})(f\circ \tensor \sigma_i\phi(\lambda_i))(\tensor_i[\op{P}(\lambda_i^{-1})(p_i)\tensor id_{v_i}])  \\ 
& = [\op{P}(g)(\tensor_i \op{P}(\lambda_i^{-1})(\tensor_ip_i))\tensor \sigma] 
 = [\op{P}(g\circ \tensor_i \lambda_i^{-1})(\tensor_i p_i)\tensor \sigma] \\
& = [\op{P}(g^\prime)(\tensor_i p_i)\tensor \sigma] 
= L^!(\op{P})(f\circ\tensor_i\sigma_i)(\tensor_i[p_i\tensor id_{v_i}]) \\
& = L^!(\op{P})(f)(\tensor_i[p_i\tensor \sigma_i]) 
\end{align*}
as desired.

Next we check that the definition is independent of any choice of decomposition of $f=\sigma\circ\phi(g)$ as above.  To this end suppose $f=\sigma_1\circ\phi(g_1)=\sigma_2\circ\phi(g_2)$ for $g_j\colon \tensor_{i=1}^nw_i\to w_0$ and $\sigma_j\in Aut(v_0)$; $j=1,2$.  Then by the factorization axioms (see Definition $\ref{admisdef}$) we know that there exists $\lambda\in Aut(w_0)$ such that $\phi(\lambda)=\sigma_2^{-1}\sigma_1$.  Thus $\phi(\lambda g_1) =\phi(g_2)$ and hence faithfulness implies $\lambda g_1=g_2$.  Thus on $Aut(w_0)$ coinvariants we have:
\begin{equation*}
[\op{P}(g_1)(\tensor_i p_i)\tensor \sigma_1] = [\op{P}(\lambda^{-1} g_2)(\tensor_i p_i)\tensor \sigma_2\phi(\lambda)]=[\op{P}(g_2)(\tensor_i p_i)\tensor \sigma_2]
\end{equation*}
as desired.

Finally it remains to check $L^!(\op{P})$ is a functor: i.e.\ that it respects composition of morphisms.  Let $f\colon X=\tensor_{i=1}^n v_i\to v_0$ as above and let $h_i\colon \tensor_{j\in J_i} \tensor u_j \to v_i$ and $h=\tensor h_i$, and we will check the composition $f\circ h$ on the exterior summand indexed by $Y:=\tensor_{j\in J} y_j$ where $\phi(y_j)=u_j$ for $J=\sqcup_{i}^n J_i$ and by $y_0:=w_0\stackrel{\phi}\mapsto v_0$.

We will check for $q_j\in\op{P}(y_j)$ and $\tau_j\in Aut(u_j)$ that;
\begin{equation}\label{compeq}
L^!(\op{P})(f\circ h)(\tensor_j [q_j\tensor \tau_j])= L^!(\op{P})(f)(L^!(\op{P})(h)(\tensor_j [q_j\tensor \tau_j]))
\end{equation}

And if we define $h^\prime = h\circ \tensor_j \tau_j$, this reduces to checking:
\begin{equation}\label{compeq2}
L^!(\op{P})(f\circ h^\prime)(\tensor_j [q_j\tensor id_{u_j}])= L^!(\op{P})(f)(L^!(\op{P})(h^\prime)(\tensor_j [q_j\tensor id_{u_j}]))
\end{equation}

To show this we ask: does there exists a $g\colon Y\to y_0$ and $\sigma\in Aut(v_0)$ such that $\phi(g)=\sigma f\circ h^\prime$?  First suppose there does not exist such a $g,\sigma$.  Then by definition the left hand side of $\ref{compeq2}$ is $0$.  Arguing by contradiction, it is easy to see that if the right hand side is not $0$ then there would exist such a $g,\sigma$.  And so the right hand side must also be $0$.

So now let us assume there does exist such a $g,\sigma$ with $\phi(g)=\sigma f\circ h^\prime$.  By the factorization axioms we can choose $m$ and $l$ with $h^\prime=\alpha \phi(m)$ and $\sigma f= \phi(l)\alpha^{-1}$.  By the axioms of a Feynman category we can decompose $\alpha$ (resp.\ $m$) as a $\tensor$-product of automorphisms $\alpha=\tensor_{i=1}^n\alpha_i$ (resp.\ generating morphisms  $m=\tensor_{i=1}^nm_i$). Then,

\begin{align*}
L^!(\op{P})(f)(L^!(\op{P})(h^\prime)(\tensor_j [q_j\tensor id_{u_j}])) 
& = L^!(\op{P})(f)(\tensor_i[\op{P}(m_i)(\tensor_{J_i} q_j)\tensor \alpha_i]) \\
& = L^!(\op{P})(f\circ\alpha)(\tensor_i[\op{P}(m_i)(\tensor_{J_i} q_j)\tensor id_{v_i}]) \\
& =[\op{P}(l)(\tensor_i\op{P}(m_i)(\tensor_{J_i} q_j))\tensor \sigma^{-1}] \\
& =[\op{P}(l\circ m)(\tensor_J q_j)\tensor \sigma^{-1}] = L^!(\op{P})(f\circ h^\prime)([\tensor_Jq_j\tensor id_{u_j}])
\end{align*}
as desired. \end{proof}

The construction of $L^!(-)$ is easily seen to be natural in the argument, hence this lemma tells us:

\begin{corollary}  The above construction specifies a functor $L^!\colon \mathbb{F}_1$-$\op{O}ps\to \mathbb{F}_2$-$\op{O}ps$.
\end{corollary}

Recall (Appendix $\ref{FCA}$) that a cubical Feynman category $\mathbb{F}$ admits a category of skew invariant functors $\mathbb{F}^-$-$\op{O}ps$.  If the morphism $\phi$ preserves degree then it will preserve skew-invariance.  We thus conclude: 

\begin{corollary}\label{oddLskriek} If $\phi$ is cubical, the above construction specifies a functor $L^!\colon \mathbb{F}_1^-$-$\op{O}ps\to \mathbb{F}_2^-$-$\op{O}ps$.
\end{corollary}

Finally we record a relationship between $L$ and $L^!$:
\begin{lemma}\label{lllemma}  Inclusion gives a natural transformation Id $\Rightarrow RL^!$. 
	\end{lemma}
\begin{proof} Given $\op{P}\in\mathbb{F}_1$-$\op{O}ps$ we define $\op{P}(w)\to \op{P}(w)\tensor_w k[Aut(\phi(w))]\subset L^!(\op{P})(\phi(v))$ by inclusion via the identity on the right hand side.  Since the image of $L^!(\op{P})$ on morphisms is given by pulling back, it is straightforward to show this induces a morphism $\op{P} \to RL^!(\op{P})$ in the category $\mathbb{F}_1$-$\op{O}ps$. \end{proof}

From this lemma one may construct an adjoint natural transformation $L\Rightarrow L^!$.  Although not built into the axiomatics of our formulation of the six operations, such a natural transformation often exists in examples; see Remark 4.5 of \cite{FHM}.

\subsection{The projection formula.}  Using the definition of $L^!$ given above, we now prove:

\begin{proposition}\label{projformula} For $\op{P}\in\mathbb{F}_1$-$\op{O}ps$ and $\op{O}\in \mathbb{F}_2$-$\op{O}ps$, there is a natural isomorphism:
	\begin{equation*}
\op{O}\tensor L^!(\op{P})\cong L^!(R(\op{O})\tensor\op{P}).
	\end{equation*} 
\end{proposition}
\begin{proof}  We start by comparing the underlying $\mathbb{V}_2$-modules.  Fix $v\in\mathbb{V}_2$.  On the left hand side we have:
\begin{equation*}
L^!(R(\op{O})\tensor\op{P})(v) =\ds\bigoplus_{\phi(w)=v} \left(\op{O}(v)\tensor \op{P}(w) \right)\tensor_wk[Aut(v)]
\end{equation*}
and on the right hand side we have:
\begin{equation*}
(\op{O}\tensor L^!(\op{P}))(v)=\op{O}(v)\tensor \ds\bigoplus_{\phi(w)=v} \left(\op{P}(w) \tensor_wk[Aut(v)]\right)\cong \ds\bigoplus_{\phi(w)=v} \op{O}(v)\tensor \left(\op{P}(w) \tensor_wk[Aut(v)]\right)
\end{equation*}
so we define a map $\beta_v\colon(\op{O}\tensor L^!(\op{P}))(v) \to L^!(R(\op{O})\tensor\op{P})(v)$ as the direct sum over $\phi(w)=v$ of the map
\begin{align*}
 \op{O}(v)\tensor \left(\op{P}(w) \tensor_wk[Aut(v)]\right) &\stackrel{\beta}\longrightarrow \left(\op{O}(v)\tensor \op{P}(w) \right)\tensor_wk[Aut(v)] \\ \beta(q\tensor [p \tensor \sigma])& =[(\op{O}(\sigma^{-1})(q)\tensor p) \tensor \sigma]
\end{align*}
One easily sees that $\beta_v$ is well defined.  In addition $\beta$ is $Aut(v)$ equivariant (since the induced action is diagonal on the left hand side and only on the right factor of the right hand side).  Finally one sees (eg by exhibiting an inverse) that $\beta_v$ is an isomorphism.  We extend $\beta$ strict monoidally on objects via $\beta_{\tensor v_i} := \tensor \beta_{v_i}$ and it remains to show that this morphism of $\mathbb{V}_2$-modules respects composition.  

For this it is enough to check a generating morphism and for a generating morphism $f\colon \tensor_{i=1}^n v_i\to v_0$ in $\mathbb{F}_2$ it is enough to check on each summand.  I.e.\ we fix $w_i$ with $\phi(w_i)=v_i$ and we shall show that the following diagram commutes:
\begin{equation*}
\xymatrix{ \ds\bigotimes_{i=1}^n\op{O}(v_i)\tensor \left(\op{P}(w_i) \tensor_{w_i}k[Aut(v_i)]\right) \ar[d]^{(\op{O}\tensor L^!(\op{P}))(f)} \ar[r]^{\beta_{\tensor v_i}} & \ds\bigotimes_{i=1}^n\left(\op{O}(v_i)\tensor \op{P}(w_i) \right)\tensor_{w_i}k[Aut(v_i)] \ar[d]^{L^!(R(\op{O})\tensor\op{P})(f)} \\ \op{O}(v_0)\tensor \left(\op{P}(w_0) \tensor_{w_0}k[Aut(v_0)]\right) \ar[r]^{\beta_{v_0}} & \left(\op{O}(v_0)\tensor \op{P}(w_0) \right)\tensor_{w_0}k[Aut(v_0)]}
\end{equation*}
Note, as above, if ever there does not exist such $w_i$ then both compositions are $0$.

Let $\tensor_i (q_i \tensor [p_i\tensor \sigma_i])$ be an element in the top left.  We ask the question: does there exist $\sigma\in Aut(v_0)$ and $g\colon \tensor_{i=1}^n w_i\to w_0$ such that $f\circ \tensor_i \sigma_i=\sigma\circ \phi(g)$.  If the answer is no, both compositions are $0$.  If the answer is yes then tracing this element down then right we have
\begin{equation*}
\beta_{v_0}(\op{O}(f)(\tensor q_i)\tensor [\op{P}(g)(\tensor_i p_i)\tensor \sigma])= [\op{O}(\sigma^{-1}\circ f)(\tensor q_i)\tensor \op{P}(g)(\tensor_i p_i)\tensor \sigma]
\end{equation*}
while tracing this element right then down we have:
\begin{equation*}
[ (R(\op{O})\tensor \op{P})(g)(\tensor_i(\op{O}(\sigma_i^{-1})(q_i)\tensor p_i))\tensor \sigma]
= [R(\op{O})(g)(\tensor_i\op{O}(\sigma_i^{-1})(q_i))\tensor \op{P}(g)(\tensor_i p_i)\tensor \sigma]
\end{equation*}

but $R(\op{O})(g)(\tensor_i\op{O}(\sigma_i^{-1})(q_i)) = \op{O}(\phi(g))(\tensor_i \op{O}(\sigma_i^{-1})(q_i))=\op{O}(\phi(g)\circ \tensor \sigma_i^{-1})(\tensor q_i) = \op{O}(\sigma^{-1}\circ f)(\tensor q_i)$ whence the claim. \end{proof}

\begin{remark}\label{oddprojformula}  Let us conclude this subsection by observing that if $\phi$ is cubical, the description of $\mathbb{F}^-$-$\op{O}ps\subset \hat{\mathbb{F}}$-$\op{O}ps$ as the skew-invariant functors (Appendix $\ref{FCA}$) makes it clear, simply by restriction, that the projection formula holds in the category $\mathbb{F}^-_2$-$\op{O}ps$.
\end{remark}

\subsection{The Intertwining theorem.}  We continue to consider an admissible morphism of Feynman categories $\phi\colon\mathbb{F}_1\to\mathbb{F}_2$, but we now further assume that $\phi$ is cubical.  In particular this allows us to consider the Feynman transform in the categories $\mathbb{F}_1$-$\op{O}ps$ and $\mathbb{F}_2$-$\op{O}ps$, which we denote by $D$ (after Appendix $\ref{FCA}$).

\begin{theorem}\label{itertwiningthm}  $DL^!\cong LD$

\end{theorem}
\begin{proof}  Keeping with our conventions of suppressing superscripts, (see Appendix $\ref{FCA})$, we emphasize that this is really two statements: $D^+{L^!}^+\cong L^-D^+$ and $D^-{L^!}^-\cong L^+D^-$ (after Remark $\ref{LRminusrmk}$ and Corollary $\ref{oddLskriek}$).  However, the argument that we use is completely symmetric.

Let $\op{P}\in\mathbb{F}_1$-$\op{O}ps$.  First, if we forget the differentials and merely consider the underlying graded $\mathbb{F}_2$-$\op{O}ps$, we see:
\begin{equation}\label{aniso}
D{L^!}(\op{P}) \cong F_2(l(G_1(\op{P})^\ast))\cong L(F_1(G_1(\op{P})^\ast))= LD(\op{P})
\end{equation}
where $\ast$ means linear dual.  

So it remains to check that this isomorphism respects the Feynman transform differentials on both sides.  Let $e_w\subset Iso(\mathbb{F}_1\downarrow w)$ be the full subcategory whose objects are the degree 1 morphisms of $\mathbb{F}_1$.  Recall (section 7.3 of \cite{KW}) that the differential in $D(\op{P})$ is defined as a composite of the form:
\begin{equation*}
F_1(\op{P}^\ast)(w)\to colim_{Iso(\mathbb{F}_1\downarrow w)} (lim_{e_{(-)}}\op{P}^\ast\circ s)\to F_1(\op{P}^\ast)(w).
\end{equation*}
This composite may be expressed in adjoint form, by restricting to the subspace $\op{P}(w)^\ast\subset F_1(\op{P}^\ast)(w)$, as the composite
\begin{equation*}
\op{P}(w)^\ast\stackrel{\partial_w}\longrightarrow lim_{e_w}(\op{P}^\ast\circ s)\stackrel{\Sigma}\to colim_{e_w}(\op{P}^\ast\circ s) \to colim_{Iso(\mathbb{F}_1\downarrow w)}(\op{P}^\ast\circ s)=G_1F_1(\op{P}^\ast)(w)
\end{equation*}
where $\partial_w$ is the universal map associated to the obvious cone: $(X\to w) \mapsto (\op{P}^\ast(w)\to \op{P}^\ast(X))$.  As a result, the differential in $L(D(\op{P}))$ is the adjoint of the following sequence of $\mathbb{V}_1$-modules:
\begin{equation*}
\op{P}(w)^\ast\stackrel{\partial_w}\longrightarrow lim_{e_w}(\op{P}^\ast\circ s)\to G_1F_1(\op{P}^\ast)(w)
\stackrel{id \Rightarrow RL}\longrightarrow G_1RLF_1(\op{P}^\ast)(w)
\end{equation*}
The differential on $LD(\op{P})$ transfers via the above isomorphism (equation $\ref{aniso}$) of $\mathbb{V}_2$-modules to an {\it a priori} new differential on $DL^!(\op{P})$ which is the adjoint of:
\begin{equation}\label{seqq}
\op{P}(w)^\ast\stackrel{\partial_w}\longrightarrow lim_{e_w}(\op{P}^\ast\circ s)\to G_1F_1(\op{P}^\ast)(w)
\stackrel{id \Rightarrow RL}\longrightarrow G_1RLF_1(\op{P}^\ast)(w)\cong rG_2F_2l(\op{P}^\ast)(w)
\end{equation}

On the other hand, the Feynman transform differential in $D(L^!(\op{P}))$ is defined to be the adjoint of the map:
\begin{equation*}
l\op{P}(v)^\ast\stackrel{\partial_v}\longrightarrow lim_{e_v}(l\op{P}^\ast\circ s)\stackrel{\Sigma}\to colim_{e_v}(l\op{P}^\ast\circ s) \to colim_{Iso(\mathbb{F}_2\downarrow v)}(l\op{P}^\ast\circ s)=G_2F_2(l\op{P}^\ast)(v)
\end{equation*}

To compare these two differentials we take the $l$ adjoint of the sequence of $\mathbb{V}_1$-modules in equation $\ref{seqq}$ to get a sequence of $\mathbb{V}_2$-modules:
\begin{equation*}
l(\op{P}^\ast)\stackrel{l(\partial_\ast)}\longrightarrow l(lim_{e_\ast}(\op{P}^\ast\circ s))\to ... 
\end{equation*}

Now fix $v\in \mathbb{V}_2$.  We claim that there exists a commutative triangle:
\begin{equation*}
\xymatrix{  l\op{P}(v)^\ast \ar[rr]^{l(\partial_\ast)} \ar[dr]_{\partial_v} & & \ar[dl]^{\eta} l(lim_{e_\ast}(\op{P}^\ast\circ s))(v) \\ & lim_{e_v}(l\op{P}^\ast\circ s) &  }
\end{equation*}
 
Where $\eta$ is the universal map given by realizing, for each $w$, $lim_{e_w} (\op{P}^\ast\circ s)$ as a cone over $lim_{e_v}(l\op{P}^\ast\circ s)$ as follows.  A morphism in $e_v$ is given by a triangle $X\stackrel{\sigma}\to X \stackrel{f}\to v$ where $|f|=1$ and $\sigma\in Aut(X)$.  For each such triangle, we ask if there exists a decomposition $\phi(Y)\stackrel{\sigma}\to \phi(Y)\stackrel{f=\phi(g)}\to \phi(w)$.  If not the cone map is $0$.  If yes we map:
\begin{equation*}
lim_{e_w}(\op{P}^\ast\circ s)\stackrel{\eta_g\tensor \sigma}\longrightarrow \op{P}^\ast(Y)\tensor_{Aut(Y)} k[Aut(X)]\subset l(\op{P}^\ast)(X)
\end{equation*}
Notice this map is independent of the decomposition after passing to $Aut(Y)$ coinvariants.  Here $\eta_g$ is the cone map $lim_{e_w}(\op{P}^\ast\circ s)\to \op{P}^\ast(Y)$ corresponding to $g$.

The fact that the diagram commutes is a consequence of the definition of $L^!(\op{P})$ on $Mor(\mathbb{F}_2)$.  In particular, whenever a morphism in $e_v$ can't be pulled back the cone map associated to the functor $l\op{P}^\ast\circ s$ (which is given by evaluating via $L^!(\op{P})$) is $0$, and hence agrees with $\eta$.  On the other hand, if a morphism in $e_v$ can be decomposed as above, then the evaluation at $L^!(\op{P})$ is given by pulling back and evaluating via $\op{P}$, hence the diagram commutes.

Thus, the two differentials that we would like to compare are defined as adjoints of the same sequence, and hence are equal. \end{proof}

\begin{corollary}  $DLD\sim L^!$

\end{corollary}

\begin{corollary} \label{exactcor} $ L^!$ preserves weak equivalences.
	
\end{corollary}

\subsection{Compatibility with composition.}  We end this section by observing that $L^!$ is compatible with composition of functors of Feynman categories.  Let $\mathbb{F}_1\stackrel{\phi}\to\mathbb{F}_2\stackrel{\psi}\to\mathbb{F}_3 $
be admissible morphisms of Feynman categories.  Then:

\begin{proposition} \label{comp} The composition $\psi\circ \phi$ is admissible and there is an isomorphism of functors $L_{\psi\circ\phi}^!\cong L_\psi^!L_\phi^!$.
\end{proposition}
\begin{proof}  First, the fact that admissibility is preserved under composition is easy to see.  Second, for $\op{P}\in\mathbb{F}_1$-$\op{O}ps$ the underlying $\mathbb{V}_3$-modules, are $l_\psi l_\phi  G\op{P}$ and $l_{\psi\circ\phi}G\op{P}$.  These are easily seen to be isomorphic, e.g.\ since $ r_\phi r_\psi \cong r_{\psi\circ\phi}$.  On summands this isomorphism takes:
	\begin{align*}
(\op{P}(w)\tensor_w k[Aut(v)])\tensor_vk[Aut(x)] & \cong \op{P}(w)\tensor_w k[Aut(x)] \\
\text{ by   \ \ \ } [[p\tensor \tau ]\tensor \sigma] & \mapsto [p\tensor \sigma\psi(\tau)
]	\end{align*}

We thus find that the functors $L_{\psi\circ\phi}^!(\op{P})$ and $L_\psi^!(L_\phi^!(\op{P}))$ agree on objects in $\mathbb{F}_3$, and it remains to check that they agree on generating morphisms in $\mathbb{F}_3$.  If $f\colon X\to x$ is such a morphism, in order to consider $L_{\psi\circ\phi}^!(\op{P})(f)$ we ask (Q1) does there exists $g\in Mor(\mathbb{F}_1)$ and $\sigma\in Aut(x)$ such that $f= \sigma\psi(\phi(g))$? 

On the other hand to consider $L_\psi^!(L_\phi^!(\op{P}))(f)$ we first ask (Q2) does there exist $h\in Mor(\mathbb{F}_2)$ and $\sigma^\prime\in Aut(x)$ such that $f= \sigma^\prime\psi(h)$? And if the answer is yes we subsequently ask (Q3) does there exists $g^\prime \in Mor(\mathbb{F}_2)$ and $\tau\in Aut(t(h))$ such that $h= \tau\phi(g^\prime)$?

The answer to (Q1) is yes if and only if the answers to both (Q2) and (Q3) are yes.  In particular given $\sigma^\prime, \tau,h,g^\prime$ as in (Q2) and (Q3), we let $g=g^\prime$ and $\sigma=\sigma^\prime\phi(\tau)$.  Notice this agrees with the isomorphism on the underlying $\mathbb{V}_3$-modules.  If the answers are yes then we define the image of $f$ in both cases by pulling back and evaluating $\op{P}(g)=\op{P}(g^\prime)$. If the answers are no then both sides evaluate to zero. \end{proof}

\section{The functor $R^!$.}  We continue to posit an admissible morphism $\phi\colon\mathbb{F}_1\to\mathbb{F}_2$.  In this section we will construct a left adjoint to the functor $L^!_\phi$ which will be denoted $R^!_\phi$ or just $R^!$ if $\phi$ is fixed.  Working by analogy to the classical example of six operations formalism in sheaf theory, we first argue such a $R^!$ exists, and then proceed to give a description of it.

\subsection{Existence of $R^!$.}

\begin{lemma}\label{adjlem}  The functor $L^!$ has a left adjoint, which we call $R^!$.
\end{lemma}
\begin{proof}
	We collect the following four facts. First, $lG_1=G_2L^!$ (equation $\ref{ldualeq}$).  Second, $l$ is a right adjoint (Lemma $\ref{amblem}$).  Third, the functors $G_1$ and $G_2$ are monadic, meaning that there is an equivalence of categories between algebras over the monad $G_iF_i$ and the source of $G_i$ (Theorem 1.5.6 of \cite{KW}).  Fourth, the category $\mathbb{F}_1$-$\op{O}ps$ has coequalizers (Lemma 8.1.5 of \cite{KW}).  
	
	The adjoint lifting theorem for monadic functors (Theorem 4.5.6 of \cite{FB}) states that when the four hypotheses above are satisfied, the functor $L^!$ has a left adjoint. \end{proof}

\subsection{Properties of $R^!$.}  
We can immediately record several properties of $R^!$ which follow from the adjunction $(R^!,L^!)$ and our work in Section $\ref{Lsec}$.  For example from Proposition $\ref{comp}$ we see:

\begin{lemma} 
	Let $\mathbb{F}_1\stackrel{\phi}\to\mathbb{F}_2\stackrel{\psi}\to\mathbb{F}_3 $
	be admissible morphisms of Feynman categories.  Then there is an isomorphism of functors $R_{\psi\circ\phi}^!\cong R_{\phi}^!R_{\psi}^!$.
\end{lemma}

In addition, the isomorphism $lG_1\cong G_2L^!$ and adjointness implies:

\begin{lemma}\label{R! lem}
	$R^!F_2\cong F_1r$.   In particular $R^!$ preserves free objects.
\end{lemma}

We now impose the condition that $\phi$ is cubical, and we would like to use the intertwining theorem (Theorem $\ref{itertwiningthm}$) to show that $R^!$ satisfies an intertwining relationship with $D$ and $R$.  This requires one additional ingredient which we draw from \cite{KW}; that morphisms from the Feynman transform can be described as solutions to an associated master equation.

\begin{theorem}\label{rint}
	$R^!D\cong DR$.
\end{theorem}
\begin{proof}  
	We use theorem 7.5.3 of \cite{KW} which says that morphisms from the Feynman transform are equivalent to solutions of a so-called master equation.  In particular, there is a natural isomorphism of sets:
	\begin{equation*}
	Hom(D(\op{O}), \op{Q})\cong ME(lim(G(\op{O}\tensor\op{Q})))\cong  Hom(D(\op{Q}), \op{O})
	\end{equation*}	
	Thus (after Theorem $\ref{itertwiningthm}$) natural isomorphisms:
	\begin{align*}
Hom_1(R^!D(\op{O}),\op{P})\cong  
Hom_2(D(\op{O}),L^!(\op{P}))\cong 
Hom_2(DL^!(\op{P}),\op{O}) \\
\cong Hom_2(LD(\op{P}),\op{O})\cong Hom_1(D(\op{P}),R(\op{O}))\cong Hom_1(DR(\op{O}),\op{P})	\end{align*}	\end{proof}

\begin{corollary}
	$DR^!D\sim R$
\end{corollary}

As a consequence of this corollary, we see that, while $R^!$ doesn't preserves weak equivalences in general, it does preserve weak equivalences after applying $D$, i.e. after moving to cofibrant domains.  This suggests that $R^!$ may be a left Quillen functor, which it is.  In particular we saw above that $L^!$ preserves weak equivalences and it can also easily be seen to preserve fibrations (which are levelwise surjections).  Hence $(R^!, L^!)$ is a Quillen adjunction.

\subsection{Description of $R^!$} 
Above, we defined the functor $R^!$ as an adjoint which was seen to exist by general principles.  In the case when $\phi$ is cubical, it is possible to give a more hands on description of the functor $R^!$.

To begin, for $E\in\mathbb{V}_2$-Mods we can consider the adjoint of the isomorphism on $R^!(F_2(E))\cong F_1(rE)$, which is a map $F_2(E)\to L^!(F_1(rE))$.  Taking $G_2$ of both sides, using the isomorphism $lG\cong GL^!$, and the ambidexterity lemma (Lemma $\ref{amblem}$) we get a natural transformation, which we call $\sigma$:
\begin{equation}\label{sum}
\sigma\colon rG_2F_2=G_1RF_2\Rightarrow G_1F_1r
\end{equation}

We use the notation $\sigma$ for ``sum''.  In particular, $\sigma$ can be informally described as summing over pre-images of a morphism.  The functor $R^!$ then has the following description:

\begin{proposition}\label{quad}  Let $\op{O}\cong F_2(E)/\langle S \rangle$ be a presentation of $\op{O}$.  Then $R^!(\op{O})\cong F_1(rE) / \langle \sigma(S) \rangle$.
\end{proposition}
\begin{proof}
	We are obliged to show that $F_1(rE) / \langle \sigma(S) \rangle$ satisfies the requisite adjointness, i.e. that there are natural isomorphisms:
	\begin{equation*}
	Hom(F_2(E)/\langle S \rangle, L^!(\op{P})) \cong 
	Hom(F_1(rE) / \langle \sigma(S) \rangle, \op{P})
	\end{equation*}
	By Lemma $\ref{R! lem}$ we know there is a natural isomorphism:
	\begin{equation*}
	Hom(F_2(E), L^!(\op{P})) \cong Hom(F_1(rE), \op{P})
	\end{equation*}	
	and we consider a pair of morphisms $(\psi_1\stackrel{\cong}\leftrightarrow\psi_2)$ which correspond via this isomorphism.  Then it is a straightforward exercise to show that the following diagram commutes:
	\begin{equation}\label{asquare}
	\xymatrix{rGF_2(E)\ar[rr]^{rG(\psi_1)} \ar[d]^{\sigma} & &rlG(\op{P}) \ar[d]^{rl\Rightarrow id} \\
		GF_1(rE) \ar[rr]^{G(\psi_2)} & & G(\op{P})
	}
	\end{equation}	
	
	To prove the proposition, it is now enough to show that  $\psi_1(s)=0$ for every $s\in S$ if and only if $\psi_2(\sigma(s))=0$ for every $s\in S$.  For this, it is enough to consider $s \in S_v\subset G_2F_2(E)(v)$ for a given $v\in\mathbb{V}_2$.  If $\psi_1(s)=0$, then by commutativity of the diagram $\ref{asquare}$ we see that 
	$\psi_2(\sigma(s))=0$.  So it suffices to prove the other implication, for which we use contraposition.

		Suppose $\psi_1(s)\neq 0$.  First, we note that this implies $v\in im(\phi)$ and we fix $w$ such that $\phi(w)=v$ and $\psi_1(s)|_w\neq 0$.  We then trace $\psi_1(s)|_w$ downward in diagram $\ref{asquare}$ by revisiting the counit of the adjunction $(r,l)$ from the proof of Lemma $\ref{amblem}$ above.  If we restrict to the $w$-summand and evaluate at $G(\op{P})$, the counit $rl\Rightarrow id$ gives us an $Aut(w)$-equivariant map $\op{P}(w)\tensor_{Aut(w)} k[Aut(v)]\to \op{P}(w)$ of the form:
	\begin{equation*}
	\sum_{g\in Aut(v)} [p_g\tensor g] \mapsto \sum_{h\in Aut(w)} h p_{\phi(h)}.
	\end{equation*}
	Since $\phi$ is faithful, $Aut(w)$ is isomorphic to a subgroup of $Aut(v)$ via $\phi$, and we may choose representatives of its right cosets: $id=g_0,g_1\cdc  g_m$.  Then each $[p_g\tensor g]$ may be rewritten in the form $[hp_{\phi(h)g_i}\tensor g_i]$ for some $i$. Thus if we define $p_i:= \sum_{h\in Aut(w)} hp_{\phi(h)g_i}$ then
	\begin{equation*}
		\sum_{g\in Aut(v)} [p_g\tensor g]= 	\sum_{i=1}^m [p_i\tensor g_i].
	\end{equation*}
	Therefore we may write $0\neq \psi_1(s)|_w=\sum_{i=1}^m [p_i\tensor g_i]$ for some $p_i\in \op{P}(w)$, and conclude there exists $j$ such that $p_j\neq 0$.  Then, from the commutativity of diagram $\ref{asquare}$, we may conclude $\psi_2(\sigma(g_j^{-1} s))|_w= p_j\neq 0$ and since $g_j^{-1} s \in S$ this proves the claim.\end{proof}

By assumption (see appendix $\ref{FCA}$) the morphism $\phi$ respects degrees, and so $\sigma$ respects the induced grading on free objects.  It follows then from the above proposition that $R^!$ preserves quadratic objects.

\section{Quadratic intertwining and Manin products.} \label{sec:quad}

The notion of duality on $\mathbb{F}$-$\op{O}ps$ that we thus far considered is bar-cobar duality.  However, if we restrict our attention to sub-categories of quadratic objects, there is another candidate namely quadratic duality, denoted $(-)^!$.  It is natural to ask if there exist parallel quadratic intertwining statements which can be seen as a formal consequence of a six operations formalism on the quadratic categories $\mathbb{F}_\op{Q}$-$\op{O}ps$ (as defined in Appendix $\ref{FCA}$).  Such a result would give a description of $L^!$ and $R^!$ on these subcategories in terms of intertwining with $(-)^!$.  Of course, such a result would require a symmetric monoidal structure for these categories, but there is a candidate given by a generalization of Manin products.

It turns out, however, that implementing this idea is only possible in certain cases.  The first problem is that the right adjoints $L^!$ and $R$ do not preserve quadratic objects in general.  Additionally, a monoidal unit for Manin products requires restricting to binary quadratic objects.  These considerations lead us to present the results of this section as a specialization of the general story told above.

For the remainder of this section we let $\phi$ be a map of cubical Feynman categories, and refer to Appendix $\ref{FCA}$ for the notion of quadraticity at this level of generality.  Notice that in this context $L$ will preserve quadraticity, since it preserves freeness and weight.  

\subsection{Quadratic Intertwining of $L^!$} \label{6.1}

Let us first consider when $L^!$ preserves quadraticity.  Let $\op{P}=F_1(A)/\langle T \rangle$ be a quadratic presentation of $\op{P}\in\mathbb{F}_1$-$\op{O}ps$, and consider the following diagram:
\begin{equation*}
\xymatrix{ F_2(lA)\ar[dr]_\epsilon \ar@{.>}[rr]^{\alpha} && L^!(\op{P}) \\ & \ar[ur]_{L^!(\pi_1)} L^! F_1(A)   & }
\end{equation*}

Here the map $\epsilon\colon F_2(lA)\cong L(F_1(A))\to L^!(F_1(A))$ is adjoint to the natural transformation given in Lemma $\ref{lllemma}$ and is a level-wise surjection (as we will see below).  The map $\alpha$ is defined to be the composite.  

We would like to show that $L^!(\op{P})$ has a quadratic presentation as $L^!(\op{P})\cong F_2(lA)/ ker(\alpha)$.  Since $\alpha$ is a level-wise surjection, it remains to show that $ker(\alpha)$ is generated in weight $1$.  Clearly $ker(L^!(\pi_1))=ker(l(\pi_1))$ is generated in weight $1$, thus it remains to know if $ker(\epsilon)$ is generated in weight $1$.  For this we define:

\begin{definition}\label{lqp} $\phi$ is left quadratic preserving if $ker(\epsilon)$ is generated in weight $1$.
\end{definition}


\begin{lemma}\label{lqlem}  If $\phi$ is left quadratic preserving, $L^!$ preserves quadratic objects.
\end{lemma}
\begin{proof}  To show $L^!(\op{P})\cong F_2(lA)/ ker(\alpha)$ is a quadratic presentation, it remains only to show that $\epsilon$ is a level-wise surjection.  We analyze $\epsilon$ in some detail for future use.
	
Considering $F_2\circ l\colon \mathbb{V}_1$-$Mods\to \mathbb{F}_2$-$\op{O}ps$, then $F_2\circ l(A)=Lan_{\phi\circ \iota_1}(A)$.  Thus every element in $F_2(lA)(v)$ is represented by a finite sum of pairs of the form $(g,\vec{a})$ where $g\in Mor(\mathbb{F}_2)$ of the form $g\colon \phi(X)\to v$ for some $X\in Ob(\mathbb{F}_1)$, and $\vec{a}\in A(X)$.  Two such pairs represent the same element if and only if they are related by some $\lambda\in Aut(X)$ as $(g,\lambda(\vec{a}))=(g\phi(\lambda),\vec{a})$.  We denote equivalence classes of such pairs as $[g,\vec{a}]$.  We may thus write a generic non-zero element in $F_2(lA)(v)$ as $\sum_{j}[g_j,\vec{a}_j]$ where each summand is non-zero, and where $g_i\neq g_j\phi(\lambda)$ for any such $\lambda$, when $i\neq j$.

Given such a class $[g,\vec{a}]$ we can ask, does there exist $f\colon X\to w$ with $\phi(w)=v$, and a $\sigma\in Aut(v)$ such that $g=\sigma\phi(f)$?  The answer to this yes-or-no question is independent of the choice of representative, and so it determines a
direct sum decomposition $F_2(lA)(v)\cong F^{Y}_2(lA)(v)\oplus F^N_2(lA)(v)$, by letting $F^{Y}_2(lA)(v)$ (resp. $F^{N}_2(lA)(v)$) be the span of those classes in which the answer is yes (resp. no).  Then the map $\epsilon$ sends $F^N_2(lA)(v)$ to $0$ and $[\sigma\phi(f),\vec{a}]\in F^Y_2(lA)$ to $[[f,\vec{a}]\tensor \sigma]$.

We now argue that $ker(\epsilon)(v)\cong F^N_2(lA)(v)$, and it suffices to show that $\epsilon$ is injective when restricted to the summand $F^Y_2(lA)(v)$.  Let us show this by contradiction by setting $\epsilon(\sum_{j\in J}[g_j,\vec{a}_j])=0$ for a nonzero generic element in $F^Y_2(lA)(v)$ in the above form, and with $g_j=\sigma_j\phi(f_j)$, so that $\sum_j[[f_j,\vec{a}_j]\tensor \sigma_j]=0$.  This sum is sub-indexed over the targets of the morphisms $f_j$, and, for such a target $w\mapsto v$, in turn over orbits of the rights $Aut(w)$ action on $Aut(v)$.  Thus there exists a nonempty subset $J^\prime\subset J$ with
$\sum_{j\in J^\prime}[\lambda_jf_j,\vec{a}_j]=0$ in $F_1(A)(w)$ for some $\lambda_j\in Aut(w)$, such that $\sigma_j\phi(\lambda_j^{-1})$ is independent of $j\in J^\prime$.  Using the description of $F_1$ as a left Kan extension, it follows that there must be distinct indicies $i,l\in J^\prime$ such that $\lambda_if_i=\lambda_lf_l\nu$ for some $\nu\in Aut(s(f_l))=Aut(s(f_i))$.  Therefore, $g_i=g_l\phi(\eta)$, contradicting our choice of input.  \end{proof}

\begin{theorem}\label{leftint}
	Let $\phi$ be left quadratic preserving and $\op{P}\in\mathbb{F}_1$-$\op{O}ps$ quadratic.  Then $L^!(\op{P})\cong L(\op{P}^!)^!$.
\end{theorem}

\begin{proof}  

	Let $\op{P}=F_1(A)/\langle T \rangle$.  The generators of $L(\op{P})$ and $L^!(\op{P}^!)^!$ are $lA\cong (lA^\ast)^\ast$, and it suffices to show that this isomorphism identifies their relations. For this we will use the following diagrams:   
\begin{align*}
\xymatrix{ F_2(lA)\ar[dr]_\cong \ar@{.>}[rr]^\pi && L(\op{P}) \\ & \ar[ur]_{L(\pi_1)} LF_1(A)   & }
& \ \ \ \ \ \ \ &	
\xymatrix{ F_2(lA^\ast)\ar[dr]_\epsilon \ar@{.>}[rr]^{\alpha} && L^!(\op{P}^!) \\ & \ar[ur]_{L^!(T^\perp\mapsto 0)} L^! F_1(A^\ast)   & }
\end{align*}

	Considering these diagrams level-wise, the relations in $L(\op{P})(v)$ are given by $ker(\pi)$.  On the other hand, the relations in $L^!(\op{P}^!)$ are given by $ker(\alpha)$.  Using our finiteness assumption (Assumption $\ref{finite}$) we may identify $F_2(lA^\ast)\cong F_2(lA)^\ast$.  It then remains to show that
	$ker(\pi)^\perp\cong ker(\alpha)$.  Under the left quadratic preserving hypothesis this reduces to showing 
	 $ker(\pi)^\perp_1\cong ker(\alpha)_1$, where subscript denotes morphism degree, and where $ker(\pi)^\perp_1:= (ker(\pi)_1)^\perp$.

We continue to use the direct sum decomposition $F_2(lA)(v)=F^{Y}_2(lA)(v)\oplus F^N_2(lA)(v)$ as in the proof of Lemma $\ref{lqlem}$, and likewise for $F_2(lA^\ast)(v)$.  Fix $Z\in ker(\pi)_1$ and $Q\in ker(\alpha)_1$.  We endeavor to show $Q(Z)=0$.  Using the direct sum decomposition we may write $Q(Z)=(Q_Y+Q_N)(Z_Y+Z_N)= Q_Y(Z_Y)+Q_N(Z_N)$.  We then argue that $Z\in ker(\pi)_1\Rightarrow Z_N=0$.

 To see this, suppose that $[g,\vec{a}]$ is a term appearing in $Z_N$.  Note that $\pi([g,\vec{a}])$ can be described by taking the inclusion of generators $\vec{a}\in A(X)\hookrightarrow \op{P}(X)$ and applying the co-cone map $\op{P}(X)\to L(\op{P})(v)$ corresponding to $g$.  The image of $\vec{a}$ in $L(\op{P})(v)$ is identified with $\op{P}(f)(\vec{a})$ precisely when there is a factorization of $g$ of the form $\phi(X)\stackrel{\phi(f)}\to\phi(Y)\stackrel{\sigma}\to v$.  But by assumption, $|g|=1$, so either $|f|=1$ and $|\sigma|=0$ or $|f|=0$ and $|\sigma|=1$.  The former case is excluded by the assumption that $[g,\vec{a}]$ appears in $Z_N$, so the latter case must apply.  But this means that $f$ is an isomorphism, and hence the corresponding identification was already made in the preimage $L(F_1(A))(v)\to L(\op{P})$.  In particular we conclude that $\pi$ is injective on the degree $1$ summand of $F_2^N(lA)$.
The description of $\pi$ above makes clear that the image of the degree $1$ summands of $F_2^N(lA)$ and $F_2^Y(lA)$ are disjoint, hence $Z\in ker(\pi)_1\Rightarrow Z_N=0$.

 So we have $Q(Z)= Q_Y(Z_Y)$.  We may write
 \begin{equation*}
 Z_Y= \sum_{j\in J}[\sigma_j\phi(f_j),\vec{a}_j] \ \ \ \text{ and } \ \ \
 Q_Y= \sum_{i\in I}[\tau_i\phi(h_i),\vec{\eta}_i] 
 \end{equation*}
as above (so $\vec{a}_j\in A(s(f_j))$ and $\vec{\eta}_j\in A(s(h_j))^\ast$, $\sigma_j$ and $\tau_i$ automorphisms).  The natural transformation $lG\Rightarrow GL$ applied to the projection $F_1(A)\to\op{P}$ gives us the commutative square,
\begin{equation*}
\xymatrix{ lGF_1(A)(v) \ar[d] \ar[rr]^{lG(\pi_1)} & & lG(\op{P})(v) \ar[d] \\  GLF_1(A)(v) \ar[rr]^{GL(\pi_1)} & & GL(\op{P})(v)  }
\end{equation*}
whose down arrows are injective.  The preimage of $Z$ under the left downward arrow is $\hat{Z}:= \sum_j[f_j,\vec{a}_j]\tensor \sigma_j\in lGF(A)(v)$.  The fact that $Z\in ker(\pi)$ tells us that $lG(\pi_1)(\hat{Z})=0$.  Partition the indexing set $J$ into subsets $J_l$ by saying two indices are in the same block if and only if the corresponding terms satisfy $t(f_j)=t(f_i)$ and $\sigma_j=\sigma_i\phi(\lambda)$ for some $\lambda\in Aut(t(f_j))$.  The fact that $lG(\pi_1)(\hat{Z})=0$ implies that for each $l$, $\sum_{j\in J_l} [f_j,\vec{a}_j]\tensor \sigma_j$ maps to $0$ in the diagram.  Hence $\sum_{j\in J_l} [f_j,\vec{a}_j]\in T$.

Returning to $Q$ we note that $\epsilon(Q_N)=0$, and thus $Q_Y\in ker(\alpha)$.
The description of  $Q_Y= \sum_{i\in I}[\tau_i\phi(h_i),\vec{\eta}_i]$ allows us to write $\epsilon(Q_Y)= \sum_i [h_i,\vec{\eta}_i]\tensor \tau_i \in lG(F_1(A^\ast))(v)$, and we thus conclude this element is in the kernel of the map $lG(F_1(A^\ast))(v)\to l(G\op{P}^!)$.  As above, we appeal to the direct sum decomposition of this space to conclude that if we partition $I$ into blocks $I_k$ whose corresponding terms satisfy $t(h_j)=t(h_i)$ and $\tau_j=\tau_i\phi(\lambda)$ for some $\lambda\in Aut(t(h_j))$, then each $\sum_{i\in I_k} [h_i,\vec{\eta}_i]\tensor \tau_i$ will be in the kernel of $lG(T^\perp\mapsto 0)$.  From this we conclude that $\sum_{i\in I_k} [h_i,\vec{\eta}_i]\in T^\perp$.

Since the natural transformation $lG\Rightarrow GL$ is a level-wise injection, we may evaluate $Q_Y(Z_Y)$ by restricting to the pairing $lGF(A)(v)\tensor lGF(A^\ast)(v)\to k$; i.e.\ by evaluating $\epsilon(Q_Y)(\hat{Z})$.  Both terms are sums partitioned into blocks corresponding to a choice of a preimage $w$ of $v$ and a subsequent $Aut(w)$ orbit within $Aut(v)$.  Two different blocks of this partition pair to $0$, whereas the diagonal choices pair something in $T$ with something in $T^\perp$.  We thus conclude $Q_Y(Z_Y)=0$.

We have now shown that if $Q\in ker(\alpha)_1$ then $Q(Z)=0$ for every $Z\in ker(\pi)_1$, from which we conclude that $ker(\alpha)_1\subset ker(\pi)^\perp_1$.  To show the converse, we work by contraposition.  If $Q\not\in ker(\alpha)$ then $\epsilon(Q)=Q_Y\neq 0$. If we write $Q_Y= \sum_{i\in I}[\tau_i\phi(h_i),\vec{\eta}_i]$, then we can once again partition this sum over sources of the $h_i$ and subsequent $Aut(s(h_i))$ orbits of $v$.  Picking representatives of the orbits we can write each block as $\sum_{i\in I_j}[\tau\phi(h_i),\eta^\prime_i]$ for some $\tau\in Aut(v)$.  We may also assume that for distinct $i,l\in I_j$ we never have $h_i= h_l\lambda$ for $\lambda\in Aut(s(h_i))$, since otherwise we could combine these terms.  Under these assumptions, the functionals appearing in the sum $Q_Y$ have disjoint support.  

Now, since $Q_Y\not\in ker(\alpha)$, there must exist $i$ such that $(h_i,\eta^\prime)\not\in T^\perp$.  Let $s(h_i)=X$.  Then $\eta^\prime\in A(X)$, so we may choose $\vec{a}\in T_X$ such that $\eta^\prime(\vec{a})\neq 0$.  Then $Z:=[\tau\phi(h_i),\vec{a}]\in ker(\pi)$, but $Q(Z)=\eta^\prime(\vec{a})\neq 0$.  Hence $Q\not\in ker(\pi)^\perp$. \end{proof}

\begin{corollary}  Under the conditions of Theorem $\ref{leftint}$,  $DLD(\op{P})\sim L(\op{P}^!)^!$.
\end{corollary}

This result may come as a bit of an initial surprise considering that in general $L$ preserves neither weak equivalences nor Koszulity.  To understand the relationships between these notions we consider the natural maps:
\begin{equation*}
 DLD(\op{P}) \longleftarrow  DL(\op{P}^!)  \longrightarrow L(\op{P}^!)^!
\end{equation*}
In general, if $\op{P}$ is Koszul and $L$ preserves weak equivalences then the left arrow is a weak equivalence.  If $\op{P}$ is Koszul and $L$ preserves Koszulity then the right arrow is a weak equivalence.  But when $\phi$ is left quadratic preserving, the ends of this zig-zag are always weakly equivalent.
In particular, if $\op{P}$ is Koszul and $L$ is exact then $L(\op{P})$ and $L^!(\op{P})$ are Koszul.

\subsection{Quadratic Intertwining of $R^!$}

Recall that $R$ does not always preserve quadratic objects.  Thus we introduce the mirror image of the terminology of Definition $\ref{lqp}$, which will be subsequently required to state the quadratic intertwining theorem.  It involves the natural transformation $F_1r\Rightarrow RF_2$ from Lemma $\ref{ntlem}$ above.

\begin{definition}
	We say the morphism $\phi$ is right quadratic preserving if the natural transformation $F_1r\Rightarrow RF_2$ is a level-wise surjection whose kernel is generated in weight 1.  
\end{definition}
The terminology reflects the fact that if $\phi$ is right quadratic preserving then $R$ preserves quadratic objects.  To see this, let $\op{O}=F_2(E)/ \langle S \rangle$ be quadratic and consider:
\begin{equation*}
\xymatrix{ F_1(r E)\ar[dr]_\beta \ar@{.>}[rr]^{\pi_1} && R(\op{O}) \\ & \ar[ur]_{R(S\mapsto 0)} R F_2(E)   & }
\end{equation*}
We view $R(\op{O})$ as quadratic with generators $rE$ and relations $ker(\pi_1)$. Note that under the assumptions, $\pi_1$ is a composite of two level-wise surjections and hence a level-wise surjection itself.

For example, forgetting directed structure and many inclusions are right quadratic preserving.  Forgetting planar structures is not right quadratic preserving. (See Section $\ref{examples}$.)

\begin{theorem}\label{rightint}
	If $\phi$ is right quadratic preserving and $\op{O}$ is quadratic then $R^!(\op{O})\cong R(\op{O}^!)^!$.
\end{theorem}

\begin{proof}  

The natural transformation $\beta_E\colon F_1(rE)(w)\to F_2(E)(\phi(w))$ is given level-wise simply by evaluation at $\phi$.  It follows that the linear dual $\beta_E^\ast \colon F_2(E^\ast)(\phi(w))\to F_1(rE^\ast)(w)$, given by summing over pre-images, may be identified with the natural transformation $\sigma_{E^\ast}$ (in the notation of line $\ref{sum}$).  Here we have used Assumption $\ref{finite}$ to identify $F_2(E)(\phi(w))^\ast\cong F_2(E^\ast)(\phi(w))$.  Then from Proposition $\ref{quad}$ we know the relations of $R^!(\op{O}^!)$ are $\sigma_{E^\ast}(S^\perp) = \beta_E^\ast(S^\perp)$ which is precisely the orthogonal complement of the relations of $R(\op{O})$ by the above description.
 \end{proof}

\subsection{Manin products and the quadratic six operations formalism}\label{MP}

\begin{definition} $\phi$ is called quadratic preserving if it is both left quadratic preserving and right quadratic preserving.
\end{definition}

We now fix $\phi$ to be quadratic preserving. In light of the quadratic intertwining theorems given above, it is natural to ask if there is a six operations formalism for which these theorems become formal consequences.  For this we would need an appropriate notion of monoidal product for quadratic objects. A candidate for this monoidal product is given by a generalization of operadic Manin products, and in this section we define this product and show that it satisfies the projection formula.

We must emphasize, however, that generalized  Manin projects do not generally give symmetric monoidal categories, in particular they do not generally have monoidal units.  However, if we can restrict our operations $(L,R,R^!,L^!)$ to subcategories in which the $\tensor$-unit is also a unit for this new product, then we would have a six-operations formalism on these subcategories. This is typically the case for subcategories of binary quadratic objects.

Let $E_1$ and $E_2\in \mathbb{V}$-Mods.  Notice $F$ is op-lax monoidal, owing to the fact that $G$ is strict monoidal.  In particular we can take the adjoint of,
\begin{equation*}
E_1\tensor E_2\stackrel{}\longrightarrow GF(E_1)\tensor GF(E_2)\cong G(F(E_1)\tensor F(E_2))
\end{equation*}
to produce a map $F(E_1\tensor E_2)\stackrel{\lambda}\to F(E_1)\tensor F(E_2)$.  This map can be described as splitting a labeled morphism by its labels, i.e.\ sending $[g, \tensor_i(e^1_i\tensor e^2_i)]$ to $[g, \tensor e^1_i]\tensor [g, \tensor e^2_i]$, for $g\in Mor(\mathbb{F})$ and $\tensor e_i\in E(s(g))$.  We call this morphism $\lambda_{E_1,E_2}$, or just $\lambda$ if no confusion will arise.

\begin{definition}
	Let $\op{P}\cong F(E_1)/\langle S_1\rangle$ and $\op{Q}\cong F(E_2)/\langle S_2\rangle$ be quadratic objects in $\mathbb{F}$-$\op{O}ps$.  Let $\pi$ be the composition of the following sequence:
	\begin{equation*}
	F(E_1\tensor E_2)\stackrel{\lambda}\longrightarrow F(E_1)\tensor F(E_2) \stackrel{\pi_\op{P}\tensor \pi_\op{Q}}\longrightarrow \op{P}\tensor \op{Q}
	\end{equation*}
Then we define the Manin (white) product of $\op{P}$ and $\op{Q}$ as $\op{P}\circ \op{Q} := F(E_1\tensor E_2)/  ker(\pi) $.
\end{definition}

Manin products for operads were introduced in \cite{GiK}, inspired by \cite{Manin}, and more detail can be found in \cite{LV}.  The above definition gives us Manin products for any cubical $\mathbb{F}$, and we can ask how these products are related by a morphism $\phi\colon \mathbb{F}_1\to\mathbb{F}_2$.  It is not hard to see that $R$ is strong symmetric monoidal with respect to $\circ$.  We conclude by observing:

\begin{proposition}
	Generalized Manin products satisfy the projection formula.  That is, if $\op{O}\in\mathbb{F}_2$-$\op{O}ps$ and $\op{P}\in\mathbb{F}_1$-$\op{O}ps$ then
	\begin{equation*}
	\op{O}\circ L^!(\op{P}) \cong L^!(R(\op{O})\circ \op{P})
	\end{equation*}
\end{proposition}
\begin{proof}
	Suppose $\op{P}= F_1(A)/ \langle T\rangle$ and $\op{O}= F_2(E)/ \langle S\rangle$.  Then the generators of $\op{O}\circ L^!(\op{P})$ are $E\tensor lA$ whereas the generators of $L^!(R(\op{O})\circ \op{P})$ are $l(rE\tensor A)$.  These are seen to be isomorphic by a (trivial) application of the $\tensor$-projection formula (Proposition $\ref{projformula}$) above.  We then consider the following rectangular diagram: \begin{equation*}
\xymatrix{ F_2(E\tensor lA)\ar[d]_\cong \ar[r]^{\lambda \ \ \ \ \ } &  F_2(E)\tensor L^!F_1(A) \ar[r]^{\ \ \pi_{\op{O}}\tensor L^!(\pi_{\op{P}})\epsilon} & \op{O}\tensor L^!(\op{P})  \ar[r]^{\cong}  &  L^!(R(\op{O})\tensor \op{P}) \\ F_2(l(rE\tensor A))  \ar[r]^\epsilon & L^!(F_1(rE\tensor A)) \ar[r]^{L^!(\lambda) \ \ \ }  &  L^!(F_1(rE)\tensor F_1(A)) \ar[r]^{L^!(\beta\tensor id)} & L^!(RF_2(E)\tensor F_1(A)) \ar[u]_{L^!(R(\pi_{\op{O}})\tensor \pi_{\op{P}})}}
\end{equation*}
	Note this diagram is easily seen to commute by chasing through an element.  In particular all the maps are morphisms in $\mathbb{F}_2$-$\op{O}ps$, and so it suffices to chase through a generator, eg at a generic level.  Since the relations in $\op{O}\circ L^!(\op{P})$ and $L^!(R(\op{O})\circ \op{P})$ are the kernels of the two paths of the diagram, the claim follows from its commutativity. \end{proof}

\section{Six operations formalisms.}

We have now laid the groundwork to state that morphisms of Feynman categories give rise to six operations formalisms on the basic categories and effective versions on the derived categories.  In this section we will record these statements.  We then consider the two ambidextrous specializations of this general context.  We postpone the discussion of specific examples and applications of these structures until Section $\ref{examples}$. 

\subsection{On the basic categories.}  The requisite work to prove the following theorem was all done above.  The main requirement was the construction of $L^!$ satisfying the projection formula Proposition $\ref{projformula}$.  Let us also remark that the categories $\mathbb{F}$-$\op{O}ps$ are indeed co-closed.  This follows as in \cite{BM} by applying the monadic lifting theorem, c.f.\ \cite{FB}.   

 \begin{theorem}\label{thm1}
 	Let $\phi\colon \mathbb{F}_1\to\mathbb{F}_2$ be an admissible morphism of Feynman categories.  Then $(L,R,R^!,L^!,cohom, \tensor)$ constitute a six operations formalism from $\mathbb{F}_1\text{-}\op{O}ps$ to $\mathbb{F}_2\text{-}\op{O}ps$.
 \end{theorem}

\subsection{On the homotopy categories.}

We now assume that $\phi$ is cubical and consider the categories of $\mathbb{F}$-$\op{O}ps$ with suppressed subscripts to be the disjoint union $\mathbb{F}^+\text{-}\op{O}ps\coprod \mathbb{F}^-\text{-}\op{O}ps$ (as in Appendix \ref{FCA}).

If $\phi$ is admissible then $(L,R)$ is a Quillen adjunction (see Remark $\ref{Quillenadj}$) and since all objects are fibrant, we have an induced adjunction on the homotopy categories by cofibrant replacement on the left, $(LD^2,R)$.  There is a dual adjunction on the derived categories given by $(DRD, DLD)$ and by the intertwining theorems we have $(DRD, DLD)\sim (R^!D^2, L^!)$.  To construct the $(cohom, \tensor)$ adjunction we prove:

\begin{lemma}  Let $\mathbb{F}$ be a cubical Feynman category.  Then for each $\op{P}\in\mathbb{F}$-$\op{O}ps$ there is a self-adjunction $(D(D(\op{-})\tensor\op{P}), \op{P}\tensor-)$ on the category $Ho(\mathbb{F}\text{-}\op{O}ps)$.  In particular, $Ho(\mathbb{F}\text{-}\op{O}ps)$ is a co-closed symmetric monoidal category with $cohom(\op{O},\op{P})\cong D(D(\op{O})\tensor\op{P})$.
\end{lemma}
\begin{proof}  We recall from theorem 7.5.3 of \cite{KW} that homotopy classes of morphisms from the Feynman transform are equivalent to homotopy classes of solutions to an associated master equation.  In particular we see that there is an isomorphism of sets:
	\begin{equation*}
	[D(\op{N}\tensor \op{P}),\op{Q}] \cong [D(\op{N}),\op{P}\tensor \op{Q}]	
	\end{equation*}
since both are isomorphic to $\pi_0(ME(\op{N}\tensor\op{P}\tensor\op{Q}))$.  Substituting $\op{N}= D(\op{O})$, in $Ho(\mathbb{F}\text{-}\op{O}ps)$ we have:
\begin{equation*}
Hom(D(D(\op{O})\tensor\op{P}),\op{Q})\cong [D(D(\op{O})\tensor \op{P}),\op{Q}] \cong [D^2(\op{O}),\op{P}\tensor \op{Q}]	\cong Hom(\op{O},\op{P}\tensor\op{Q}),
\end{equation*}
hence the claim. \end{proof}

\begin{theorem}\label{hoso}
The six-operations formalism of Theorem $\ref{thm1}$ passes to the homotopy categories.  Moreover, it extends to an effective six-operations formalism from $Ho(\mathbb{F}_1\text{-}\op{O}ps)$ to $Ho(\mathbb{F}_2\text{-}\op{O}ps)$ with dualizing object $D(I_2)$.
\end{theorem}
\begin{proof}  Again the work has been done above.  Observe the projection formula is easily seen to be still satisfied, since it involves only right adjoints.  The fact that $D(I_2)$ is dualizing follows immediately from the description of $cohom$ above.  Moreover we see that $cohom(D(I_2),-)=D(-)$ and so the intertwining theorems which result from this dualizing object are the expected ones. \end{proof}

\subsection{Ambidexterity: the Grothendieck and Wirthm{\"u}ller contexts}\label{WGsec}
Inspired by \cite{FHM}, let us give two specializations of this general context.

\begin{definition}
A morphism as in Theorem $\ref{thm1}$ is called proper if $L^!\cong L$ and co-proper if $R^!\cong R$.
\end{definition}

In the terminology of \cite{FHM}, what we call a proper map is said to satisfy the ``Grothendieck context'' and what we call a co-proper map is said to satisfy the ``Wirthm{\"u}ller context''.  A proper map gives rise to a triple of adjoint functors $(R^!,L,R)$ whereas a co-proper map gives rise to a triple of adjoint functors $(L,R,L^!)$.  Symmetrization is often proper, inclusion is often co-proper; examples are discussed below (Section $\ref{examples}$).  Note that the ambidexterity lemma (Lemma $\ref{amblem}$) shows that maps between trivial Feynman categories (i.e.\ those which are groupoids) are both proper and co-proper.  

We now give criteria for recognizing proper and co-proper maps.

\begin{proposition}	A morphism $\phi$ is proper if for every morphism in $\mathbb{F}_2$ of the form $f\colon \phi(X)\to v$, there exists $\sigma \in Aut(v)$ and  $g\in \mathbb{F}_1$ with $s(g)=X$ such that $f=\sigma\phi(g)$. \end{proposition}

\begin{proof}   Recall that the functor $L$ is given by left Kan extension, explicitly $L(\op{P})= Lan_\phi(\op{P})$.  So it suffices to prove that, under the stated conditions, $L^!(\op{P})$ is such a left Kan extension.
		
		The conditions of the proposition ensure that any such generating morphism $f\in Mor(\mathbb{F}_2)$ can be written as:
		\begin{equation*}
		\xymatrix{  \phi(X) \ar[rr]^{\phi(g)}  \ar[dr]_f  &  & \phi(w) \ar[dl]^\sigma\\ &   v & }
		\end{equation*}	where $\sigma$ is an automorphism.  Hence, we may compute 
		$Lan_\phi(\op{P})(v)$ by restricting to the sub category of isomorphisms $Iso(\phi\searrow v)$.  But this is the level-wise left Kan extension of the underlying $\mathbb{V}_1$-module, which is given by the functor $l$.  Hence $lG(\op{P})(v)\cong Lan_\phi(\op{P})(v)$, and hence $L^!(\op{P})$ coincides with $L(\op{P})$ on objects.  This identification respects morphisms, since in both instances the image of a morphism $f$ is given by factoring as above and evaluating with $\op{P}(g)$. \end{proof}

\begin{proposition}\label{coprop}
	A morphism $\phi$ is co-proper if for every morphism in $\mathbb{F}_2$ of the form $f\colon Y\to \phi(w)$, there exists $\sigma\in Aut(Y)$ and $g\in\mathbb{F}_1$ with target $t(g)=w$ and $f=\phi(g)\sigma$.
\end{proposition}
\begin{proof}  

As was pointed out in \cite{KW}, the right Kan extension $Ran_\phi(-)$ is a priori only a lax monoidal functor, but is strong monoidal in some examples.  We will show the conditions of the proposition are sufficient for this to be the case, and in particular $L^!\cong Ran_\phi$ under these conditions, from which the claim will follow.
	
	To see this, fix $\op{P}\in\mathbb{F}_1$-$\op{O}ps$.  We first observe that for $v\in \mathbb{V}_2$, using the description $Ran_\phi(\op{P})(v)=lim_{v\searrow \phi}\op{P}\circ t$, it is easy to see that $Ran_\phi(\op{P})(v)\cong L^!(\op{P})(v)=lG\op{P}(v)$.  In particular both sides are zero unless $v\in im(\phi)$.  We then check that for $Y=\tensor v_i$ we have $Ran_\phi(\op{P})(Y)\cong \tensor_i L^!(\op{P})(v_i)$.  First, if $Y\not\in im(\phi)$ then both sides are zero, since in particular our conditions would imply that there are no morphisms $Y\to \phi(w)$.  So we assume $Y\in im(\phi)$.
	
	Then, under the stated conditions, any object $f\in Y\searrow \phi$ can be completed to a diagram:
	\begin{equation*}
	\xymatrix{ & \ar[dr]^f \ar[dl]_\sigma Y & \\  \phi(X) \ar[rr]^{\phi(g)}&  & \phi(w)  }
	\end{equation*}	
	where $\phi(X)=Y$ and $\sigma\in Aut(Y)$.  Therefore $lim_{Y\searrow \phi}\op{P}\circ t$ may be computed by restricting to the sub-category of isomorphisms $Iso(Y\searrow \phi)$.  But if we factor $Y=\tensor_i v_i$ then the Feynman category axioms allow us to write $Iso(Y\searrow \phi)=\times_i Iso(v_i\searrow \phi)$.  We thus conclude $Ran_\phi(\op{P})(Y)\cong \tensor_iRan_\phi(\op{P})(v_i) $ and that $L^!(\op{P})$ and $Ran_\phi(\op{P})$ coincide on objects.
	
    It is then easy to see that this correspondence respects the image of morphisms.  Having concluded that $L^!(\op{P})$ and $Ran_\phi(\op{P})$ coincide, we conclude that their respective left adjoints coincide as well, i.e.\ $R^!\cong R$.  \end{proof}

Having characterized situations in which $\phi$ is proper or co-proper, let us now consider consequences of such situations.

\begin{lemma}  Let $\phi$ be co-proper.  Then the associated natural transformation $Fr\Rightarrow RF$ is an isomorphism.  In particular $\phi$ is right quadratic preserving, $R$ preserves free objects and $R$ commutes with $D$.  Moreover $R$ is a co-closed symmetric monoidal functor: $R(cohom(-,-)) \cong cohom(R(-),R(-))$.
\end{lemma}
\begin{proof}  The first statement follows from Lemma $\ref{R! lem}$.  The second statement follows from Theorem $\ref{rint}$ and the fact that $R^!$ preserves quadratic objects.  For the final statement, we first fix $\op{O},\op{P}\in \mathbb{F}_2$-$\op{O}ps$ and $\op{Q}\in \mathbb{F}_1$-$\op{O}ps$.  We then use adjointness and the projection formula (Proposition $\ref{projformula}$) to exhibit a sequence of natural isomorphisms:
 \begin{align*}
	Hom_1(R^!(cohom(\op{O},\op{P})),\op{Q})\cong Hom_2(cohom(\op{O},\op{P}), L^!(\op{Q})) \cong 
	Hom_2(\op{O},\op{P}\tensor L^!(\op{Q})) 	
	 \\ \cong Hom_2(\op{O}, L^!(R(\op{P})\tensor \op{Q}))  \cong Hom_1(R^!(\op{O}), R(\op{P})\tensor \op{Q})  \cong Hom_1(cohom(R^!(\op{O}), R(\op{P})), \op{Q}).
	\end{align*}
We thus have a natural isomorphism of functors $R^!(cohom(-,-)) \cong cohom(R^!(-),R(-))$, which implies the final statement when $\phi$ is co-proper. \end{proof}

We say $\phi$ is exact if $L(-)$ preserves weak equivalences.  From Corollary $\ref{exactcor}$ and the fact that $L$ preserves quadratic objects we conclude:
\begin{lemma} Proper morphisms are exact and left quadratic preserving.
\end{lemma}

\section{Verdier duality is Koszul duality}\label{vdkd}

Recall that to define the Verdier dual of a sheaf on a space $X$, one looks at the map to a point, $p\colon X\to \ast$, and defines the dualizing complex $\omega_X:=p^!(k)$.  Verdier duality is then given by the functor $D(-):= RHom(-, \omega_X)$.  We may carry out the analogous construction for a Feynman category, and in this section we show that the dualizing complex $\omega_\mathbb{F}$ coincides with the dualizing object $D(I)$.  In other words, if we had not known about derived Koszul duality {\it a priori}, we could have created it via the six operations.  To be more precise the dualizing complex $\omega_\mathbb{F}$ has two ingredients: the six operations and bar-cobar duality for commutative and Lie algebras.

To begin we need the analog of a point.  It is provided by viewing the commutative operad as a Feynman category as in Example 2.9.8 of \cite{KW} as we now recall.

\begin{definition}  Let $\mathbb{C}om$ be the Feynman category with one vertex, call it $\ast$, and a unique morphism in each arity $Hom(\ast^{\tensor n}, \ast)$. Observe that a symmetric monoidal functor from $\mathbb{C}om$ is specified by an object with a commutative and associative multiplication and so $\mathbb{C}om$-$\op{O}ps_{dgVect}$ is the category of dg commutative algebras.
\end{definition}

Given a Feynman category $\mathbb{F}$ there is a unique morphism of Feynman categories $p\colon \mathbb{F}\to\mathbb{C}om$ sending the vertices of $\mathbb{F}$ to the vertex of $\mathbb{C}om$.  However the morphism $p$ is not admissible in the sense of Definition $\ref{admisdef}$, and so we can not apply our constructions to produce the adjunction $(R^!_p, L^!_p)$.  The problem being that we have no control over what the pull-back of a morphism looks like.    However, assuming that $\mathbb{F}$ is cubical, we do have some control if we sum over only degree $1$ morphisms.  The idea to do this comes from \cite{KWZ} and \cite{KW} which show summing over degree $1$ morphisms gives higher analogs of Maurer-Cartan elements which correspond to morphisms from the Feynman transform.

Recall that an odd $L_\infty$ algebra is an algebra over $\fr{s} \op{L}_\infty$.  This is equivalent to an $L_\infty$ structure on the shifted complex; in particular the algebra structure is generated by symmetric operations of degree $1$ in each arity.  We define:

\begin{definition}  Define a functor $L_p^!\colon \mathbb{F}^-$-$\op{O}ps\to  \{ \text{odd } L_\infty\text{-Algebras} \}$ by letting $L_p^!(\op{P})= \oplus_{w\in\mathbb{V}} \op{P}(w)_{Aut(w)}$ with operations $l_n:= \sum$ degree $1$ morphisms of length $n$ in $\mathbb{F}$.  That is, for each list of objects $w_1\cdc w_n; w_0$, we define $\tensor_{i\geq 1}\op{P}(w_i)\to \op{P}(w_0)$ to be the sum over degree $1$ morphisms with source  $\tensor_{i\geq 1}w_i$ and target $w_0$.  The finiteness of this sum is assured by Assumption $\ref{finite}$.  Since pre-composing with an automorphism preserves degree, and since we are taking the sum over all degree $1$ morphisms, this map lifts to $\times_{i\geq 1}Aut(w_i)$ coinvariants.  Finally we let $l_n$ be the linear extension of these maps. 

The fact that each $l_n$ is symmetric follows from the fact that $\mathbb{F}$ is a symmetric monoidal category.  The fact that these operations have degree $1$ follows from the fact that the input is odd and that we sum over operations of degree $1$. Finally the fact that this is an $L_\infty$ algebra follows from the fact that $\sum_{r+s=n} l_r\circ l_s$ is given by a sum of factorizations of degree $2$ morphisms into two degree $1$ morphisms.  Since $\mathbb{F}$ is assumed to be cubical, these come in pairs; since the input was assumed to be odd, these pairs have opposite sign and so the value of this sum is zero, as desired.
\end{definition}

\begin{remark}  In examples of interest this $L_\infty$ structure is often a Lie algebra (e.g.\ for operads \cite{KM}), or a dg Lie algebra (e.g.\ for modular operads \cite{Bar}).  It may carry additional structures as well,  see Table 3 of \cite{KW}.
\end{remark}

\begin{lemma}\label{compcor}  Let $\phi\colon \mathbb{F}_1\to\mathbb{F}_2$ be an admissible morphism of cubical Feynman categories. There is an isomorphism of functors $L^!_{p_2}L^!_\phi \cong L^!_{p_1}$
\end{lemma}
\begin{proof}  Here $p_1$ and $p_2$ denote the respective projections to $\mathbb{C}om$.  The fact that the underlying graded vector spaces agree follows from the proof of Proposition $\ref{comp}$.  Recall (Appendix $\ref{FCA}$) that morphisms of cubical Feynman categories are by definition degree preserving.  Hence the sum over degree $1$ morphisms in $\mathbb{F}_1$ is the same as the sum over preimages of degree $1$ morphisms in $\mathbb{F}_2$, and so the $L_\infty$ structures agree as well.\end{proof}

\begin{lemma}  The functor $L^!_p$ has a left adjoint, which we call $R^!_p$.
\end{lemma}
\begin{proof}  The proof follows as in Lemma $\ref{adjlem}$ above.  Algebras over operads can be encoded in the language of Feynman categories via indexed enrichment (Chapter 4 of \cite{KW}).  Since the $L_\infty$ operad is dimension $1$ in arity $1$, the $\mathbb{V}$-modules in the enriched Feynman category encoding $L_\infty$-algebras are just given by the underlying dg vector spaces.  Hence Lemma $\ref{amblem}$ still applies to give an adjunction $(r_p,l_p)$ between the respective categories of $\mathbb{V}$-modules.  We also recall the enriched monadicity theorem 4.1.4 of \cite{KW}, which states that the forgetful functor is still monadic in this enriched case.  We may therefore apply the adjoint lifting theorem for monadic functors (Theorem 4.5.6 of \cite{FB}) to the commutative square $l_pG= GL^!_p$ of functors $\mathbb{F}^-$-$\op{O}ps\to dgVect_k$ to prove the lemma.\end{proof}

We now define the analog of the dualizing complex to be the $\mathbb{F}^-$-operad $\omega_\mathbb{F}:=R^!_p(D(k))$.  Here $k$ is the ground field viewed as a commutative algebra and $D$ is its Feynman transform.  This is, up to a shift in degree, the cobar construction of $k$ viewed as a cocommutative coalgebra.  It may also be identified with the operadic deformation complex of the identity map $
\op{L}ie\to\op{L}ie$.

\begin{theorem} In $Ho(\mathbb{F}$-$ops)$ there is an isomorphism $D(\op{P})\cong cohom(\omega_\mathbb{F},\op{P})$.
\end{theorem}

\begin{proof}  First by adjointness we have $Hom(\omega_\mathbb{F},\op{P})\cong Hom(D(k),L^!_p(\op{P}))$.  This latter Hom is in the category of odd $L_\infty$ algebras and since $D(k)$ is quasi-free on the ground field $k$, morphisms from it are given by those elements in the target which satisfy the condition imposed by the requirement that the original morphism is dg.  From the theory of bar-cobar duality for $\op{P}_\infty$ algebras (see e.g.\ Chapter 11 of \cite{LV}) we know this condition translates to the Maurer-Cartan equation in the target. 

On the other hand from Theorem 7.5.3 of \cite{KW} we know that $MC(L^!(\op{P})) = Hom(D(I), \op{P})$.  Both of these isomorphisms are natural in $\op{P}$, and we may thus conclude $\omega_\mathbb{F}\cong D(I)$.  Applying Theorem $\ref{hoso}$ above then proves the claim.
\end{proof}

We conclude this section with an application of this theorem: it implies that $L^!$ preserves deformation complexes.  Given a morphism $\alpha\colon \omega_\mathbb{F}\to \op{P}$ we have a corresponding Maurer-Cartan element in $L_p^!(\op{P})$.  We define the deformation complex of $\alpha$ to be this $L_\infty$ algebra along with the differential formed by twisting with $\alpha$.

\begin{corollary}\label{prescor} Let $\phi\colon \mathbb{F}_1\to\mathbb{F}_2$ be an admissible morphism between cubical Feynman categories. The functor $R_\phi^!$ preserves dualizing complexes and the functor $L^!_\phi$ preserves deformation complexes.
\end{corollary}
\begin{proof}  
The first statement follows from Lemma $\ref{compcor}$.  The second statement, explicitly, means that the deformation complex associated to $\alpha\colon \omega_{\mathbb{F}_1}\to \op{P}$ is isomorphic to the deformation complex associated to the composition $\omega_{\mathbb{F}_2}\to L_\phi^!R_\phi^!(\omega_{\mathbb{F}_2})\cong L_\phi^!(\omega_{\mathbb{F}_1}) \to L_\phi^!(\op{P})$.  This follows from the first statement via adjointness.
\end{proof}


\section{Examples.}\label{examples}

In this section we will investigate the implications of the six operations formalism and our above constructions in several examples of morphisms of Feynman categories.  The sources and targets of these morphisms will be several well studied generalizations of operads, which we will briefly recall here.  In addition to the references given below, these examples are discussed in the language of Feynman categories in Chapter 2 of \cite{KW}.  In the Feynman category interpretations the morphisms are given by classes of graphs and the objects are given by lists of vertices.

\begin{center}
	\begin{tabular}{r||l|l|l}
		Type of generalized operad  & F.C. morphisms & reference & our notation  \\ \hline
		operads  & rooted trees & \cite{LV}, etc. & $\mathbb{O}$  \\
		non-$\Sigma$ operads  & planar rooted trees & \cite{LV}, etc. & $\mathbb{PO}$  \\
		cyclic operads   &  trees & \cite{GeK1} & $\mathbb{C}$ \\
		non-$\Sigma$ cyclic operads  & planar trees & \cite{GeK1} & $\mathbb{PC}$  \\
		dioperads   & directed trees & \cite{Gan} & $\mathbb{DO}$ \\
		planar dioperads   & planar directed trees & \cite{TZ} & $\mathbb{PDO}$ \\
		modular operads  & connected graphs & \cite{GeK2} & $\mathbb{M}$  \\
		
	\end{tabular}

\end{center}

\subsection{Cyclic to Modular operads}

There is a morphism of Feynman categories $\phi\colon\fccyc\to\fcmod$ from the Feynman category encoding cyclic operads to the Feynman category encoding modular operads.  On both objects and morphisms this functor is given by inclusion.

\begin{lemma}
	The morphism $\phi\colon \fccyc\to\fcmod$ is admissible and co-proper.
\end{lemma}
\begin{proof}
	Both statements follow from viewing $\fccyc \subset \fcmod$ as an inclusion, after Proposition $\ref{coprop}$.  Note the genus labeling for vertices in $\fcmod$ ensures that an object of genus zero can not be the target of a morphism of non-zero genus. \end{proof}

We thus see that $(L,R,L^!)$ is a triple of adjoint functors, all three of which are well known constructions.  $R$ forgets higher genus, $L$ is the modular envelope or modular completion, and $L^!$ is easily seen by the above description to be extension by $0$.  Explicitly, $L^!(\op{O})(g,n)= \op{O}((n))$ if $g=0$ and equals $0$ if $g\geq 1$.

Considering the intertwining theorems in this context we then see:
\begin{corollary}\label{modint}  There are isomorphisms of functors $DR\cong RD$ and $DL^!\cong LD$.
\end{corollary}

The first statement is 5.9 of \cite{GeK2}; the second statement appears to be new.  Informally, for a cyclic operad $\op{O}$ we may view both $DL^!(\op{O})$ and $LD(\op{O})$ as complexes of $\op{O}$-labeled odd (or $\fr{K}$-twisted) graphs of arbitrary genus, but no non-zero genus labels. The differential in $LD(\op{O})$ allows us to blow up an edge ``with-in'' a vertex, which is then identified with the $\op{O}$-labeled graph having this additional edge and labels determined via the 
cyclic operad structure on $\op{O}$.  The differential in $DL^!(\op{O})$ (a priori) has extra terms corresponding to blowing up loops.  But since we're taking the Feynman transform of an extension by zero, these terms must be zero as well.  Similarly, if $\op{O}$ is an odd cyclic operad, the second statement of Corollary \ref{modint} gives us an isomorphism of untwisted modular operads.

{\bf Application to graph homology.}  Let $\op{O}$ be a Koszul cyclic operad.  The complexes $DL^!(\op{O})(g,n)$ are versions of Kontsevich's original graph complexes \cite{Kont}.  Our result shows that the dg modular operad comprised of these complexes is a model for the derived modular envelope of the Koszul dual $\op{O}^!$. On the other hand, we may consider the graph complexes $DL(\op{O}^!)(g,n)$ which are of interest in their own right.  Our results allow us to build a bridge between these two families of graph complexes.

To illustrate this idea we consider the case $\op{O}=\op{C}om^!=\Sigma^{-1}\fr{s}\op{L}ie$, the odd cyclic operad encoding odd Lie algebras.  The graph complexes $DL^!(\op{O})(g,n)$ compute (in the case $n=0$) homology of the group of outer automorphisms of the free group $F_g$.  For $n\geq 1$ these ``hairy'' graph complexes compute the homology of generalizations of $Out(F_g)$ realized as boundary preserving homotopy self-maps of bouquets of $g$ circles and $n$ segments, see \cite{CKV}, \cite{CHKV}, which call these groups $\Gamma_{g,n}$.  Our results allow us to study these graph complexes via the cofiltration;
\begin{equation}\label{cofil}
DL^!(\op{O})\cong L(\op{C}_\infty)\twoheadrightarrow\dots \twoheadrightarrow L(\op{C}_5)\twoheadrightarrow L(\op{C}_4)\twoheadrightarrow L(\op{C}om)
\end{equation}
where $\op{C}_N$ denotes the operad encoding those $\op{C}_\infty$ algebras whose generating operations $\mu_n$ vanish at and above $\mu_N$ (see \cite{LV} Proposition 13.1.6).

On the other hand, as I learned from Dan Petersen, the graph complexes $DL(\op{C}om)(g,n)$ are of interest in that they are related to the moduli space of tropical curves of constant volume, which we denote $\Delta_{g,n}$ after \cite{CGP}.  In particular, $DL(\op{C}om)(g,n)$ computes the reduced cohomology of $\Delta_{g,n}$ with a shift in degree. If we apply the Feynman transform to the above cofiltration we may study the graph complex $DL(\op{C}om)(g,n)$ in genus $g\geq 1$ by trapping it in an acyclic complex.  Alternatively, we may apply the Feynman transform to the homology cofiltration to study $\Delta_{g,n}$ via the filtration $DL(\op{C}om)\hookrightarrow \dots \hookrightarrow D H_\ast(\Gamma_{g,n})$.

To offer some evidence that this is a strategy worth considering, let us use it to calculate something.

\begin{calculation}\label{calculation}  $\chi(\Delta_{1,n})=1+(-1)^{n-1}(n-1)!/2$ for $n\geq 3$.
\end{calculation}
\begin{proof}  Our starting point for this calculation is threefold.  First, we recall from \cite{CHKV} that $\chi(\Gamma_{1,n})= 2^{n-2}$.  Second we observe that $\chi(DH_\ast(\Gamma)(1,n))= \chi(DLD(\op{C}om^!)(1,n))= \chi(L^!(\op{C}om^!)(1,n))=0$.  Here the first equality follows from the fact that $H_\ast(\Gamma_{g,n})\sim L(\op{C}_\infty)(g,n)$ as complexes; but not necessarily as modular operads and hence the use of Euler characteristic, and the second equality follows from Corollary $\ref{modint}$.  Third we recall the theory of characteristics of cyclic operads (see \cite{GeK2}).  In particular we define $f$ to be the formal power series;
\begin{equation*}
f(x)=x+\ds\sum_{n\geq 2} v_{n+1}\ds\frac{x^n}{n!}
\end{equation*}
and define $p_n$ to be the coefficient of $x^{n-1}/(n-1)!$ in $f^{-1}$.  Observe that $p_n$ is a polynomial in the variables $v_3\cdc v_{n}$.  The coefficient of the monomial $v_{i_1}\cdot... \cdot v_{i_m}$ is, up to sign, the number of trees with $n$ labeled leaves having vertices of valence $i_1\cdc i_m$.  The sign is equal to $-1$ to the number of vertices.  In particular, if $\op{O}$ is a cyclic operad, and we set $dim(\op{O}((n)))=v_{n+1}$, then $p_n$ evaluates to $-\chi(D\op{O}((n)))$.  For example $p_n(1\cdc 1)=(-1)^{n}(n-2)!$, using the fact that the dimension of $\op{L}ie((n))$ is $(n-2)!$.

Let us now define the following three numbers:
\begin{equation*}
A:= p_n(\prod_{i\in J} v_i \leftrightarrow \sum_{i\in J}(2^{i-2}-1)), \
B:= p_n(\prod_{i\in J} v_i \leftrightarrow \sum_{i\in J}(2^{i-1}-i)), \
C:=  p_n(\prod_{i\in J} v_i \leftrightarrow \sum_{i\in J}(i-2)) 
\end{equation*}
The notation $\leftrightarrow$ indicates that we substitute the number on the right for the monomial $\prod_{i\in J} v_i$.

First observe that $C=(n-2)p_n(1\cdc 1)$.  This follows from the fact that for any $n$ tree, $n-2=\sum_v (|v|-2)$ where the sum is taken over the (internal) vertices of the tree and $|v|$ denotes the valence of the vertex.

Second we observe that $B=p_n(1\cdc 1)$.  To see this, consider that each monomial along with its coefficient represents a count of trees along with a way to pick a vertex and partition its adjacent flags (or half edges) into two sets such that neither set has size $1$.  For a vertex of valence $i$ there are $2^{i-1}-i-1$ ways to do this such that each set has at least 2 elements and a unique way such that one set has all the elements.  We then identify the portion of the count corresponding to $2^{i-1}-i-1$ with a count of trees having a distinguished edge by blowing up an edge in the original decorated graph to separate the two sets of flags.  Over all monomials this gives us a signed count of graphs along with a distinguished edge.  On the other hand we identify the portion of the count corresponding to the unique uniform labeling of a vertex to be a count of trees along with a distinguished vertex.  Since all graphs in question are trees, they always have one more vertex than the number of edges, and for each tree the choices of distinguished vertex compared to distinguished edges are in adjacent degrees and so have opposite sign.  Therefore, each tree appears exactly once in the final count after cancellation, hence $B=p_n(1\cdc 1)$.

Finally, we will argue that $A=\chi(DL(\op{C}om)(1,n))$.  We will use the notation
$DH_{\not\equiv 0}(\Gamma)(1,n)\subset DH_\ast(\Gamma)(1,n)$
to denote the subspace where at least one vertex is labeled by a class in $H_\ast(\Gamma)$ of non-zero (and hence positive) degree.  Since $0=\chi(DH_\ast(\Gamma)(1,n))$, we know that
\begin{equation*}
0=\chi(DH_0(\Gamma)(1,n))+\chi(DH_{\not\equiv 0}(\Gamma)(1,n))= \chi(DL(\op{C}om)(1,n))+\chi(DH_{\not\equiv 0}(\Gamma)(1,n)),
\end{equation*}
and it remains to argue that $A=-\chi(DH_{\not\equiv 0}(\Gamma)(1,n))$.

Recall that in the theory of modular operads, the total genus of a graph is the genus of its realization plus the genus labeling at each vertex. In particular a graph of total genus $1$ is either a graph whose realization is of genus $1$ along with a genus-label of $0$ at each vertex, or it is a tree along with a genus-label of $1$ at one vertex and a genus-label of $0$ at every other vertex.  Thus, the underlying graded vector space of $DH_{\not\equiv 0}(\Gamma)(1,n)$ splits into a direct sum of components indexed by graphs whose realization is genus $1$ and graphs whose realization is genus $0$.  In the former summand, one vertex in each such graph is labeled by a class in $H_{\geq 1}(\Gamma_{0,i})$.  However $H_\ast(\Gamma)(0,-)=H_\ast(\op{C}_\infty)=\op{C}om$ is concentrated in degree $0$ and so no such classes exist. In particular this summand must be $0$.  

It remains to analyze the latter summand (which we now know is the entire space) in which each graph is a tree.  Each such tree has a distinguished vertex whose genus-label is $1$, with all other vertices labeled by genus $0$.  On the other hand the vertex labels must have at least one class of positive degree.  Since $H_\ast(\Gamma_{0,i})$ is concentrated in degree $0$, the only possibility is that the distinguished vertex (genus 1 vertex) is labeled by $H_{\geq 1}(\Gamma_{1,i})$.  We thus have the following description of $DH_{\not\equiv 0}(\Gamma)(1,n)$:  it is a complex of trees with $n$ tails/leaves along with a distinguished vertex, of valence $i$ say, labeled by a class in $H_{\geq 1}(\Gamma_{1,i})$.  The differential uses the modular operad structure of $H_\ast(\Gamma)$, but since we are concerned here with the Euler characteristic, it won't play a role.  Since $\chi(H_{\geq 1}(\Gamma)(1,i))=2^{i-2}-1$, we may calculate $-\chi(DH_{\not\equiv 0}(\Gamma)(1,n))$ as an alternating sum of counts of trees along with a distinguished vertex $v$ weighted by $(2^{|v|-2}-1)$.  The theory of characteristics of cyclic operads, summarized above, tells us that this is precisely the quantity $A$.

To conclude we observe that $2A=B+C$ and hence
\begin{equation*}
2\chi(DL(\op{C}om)(1,n))=(-1)^{n}(n-2)!+(-1)^{n}(n-2)(n-2)!
\end{equation*}
from which the calculation follows.\end{proof}

It is important for us to point out that the final result is not new; Theorem 1.2 of \cite{CGP} proves a stronger result.  We rather wish to emphasize the technique; using Corollary $\ref{modint}$ to turn information about $\Gamma_{g,n}$ into information about $\Delta_{g,n}$ via Koszul duality.

{\bf Comparison with the topological case.}  It is shown in \cite{Gian}, via the results of \cite{LVor}, that if $\op{X}$ is a topological cyclic operad, the complex $DL^!D(C_\ast(\op{X}))$ computes the levelwise cohomology of the topological modular envelope of $\op{X}$.  Here $C_\ast$ denotes chains.  This may be compared with our result which shows that  $DL^!D(C_\ast(\op{X}))$ is itself a model for the derived modular envelope of $C_\ast(\op{X})$ in the category of chain complexes.  

The first example of interest is the topological cyclic associative operad $\op{A}s^{top}$ and its homology, the cyclic operad $\op{A}s$.  Theorem B of \cite{Gian} identifies the topological derived modular envelope of $\op{A}s^{top}$ with the moduli space of Riemann surfaces with marked windows on the boundary.  On the other hand Theorem 9.4 of \cite{GeK2} states that $DL^!(\op{A}s)$ forms a chain model for these spaces.  Corollary $\ref{modint}$ gives some unification to these two results in that it allow us to view $DL^!(\op{A}s)$ as the derived modular envelope of $\op{A}s$ in the category of chain complexes.

These examples are of further interest due to their relationship to combinatorial operads and props acting on Hochschild and cyclic chain complexes, e.g.\ arc operads \cite{KLP} (see also \cite{KSurvey} and the references there-in).  This relationship promises to be further illuminated in upcoming work of R. Kaufmann and C. Berger \cite{KB}.

\subsection{Dioperads to Cyclic operads}\label{fdex}

There is a morphism of Feynman categories $\phi\colon\fcdi\to\fccyc$ from the Feynman category encoding dioperads to the Feynman category encoding cyclic operads.  On objects and morphisms one forgets directions using the morphism $S_n\times S_m \to S_{n+m}$ given by `in then out'.  The morphism $\phi$ is neither proper nor co-proper and in this example, all four functors $(L,R,R^!,L^!)$ are distinct.  The question of precisely how cyclic operads and dioperads are related was asked in \cite{Gan}, and we submit the results of this subsection as an answer to this question.

\begin{lemma}\label{addc}
	The morphism $\phi\colon \fcdi\to\fccyc$ is admissible, quadratic preserving, and exact.
\end{lemma}
\begin{proof}
	Admissible follows from the graph description of morphisms.  In particular, a generating morphism $f\in Hom_\fccyc(\phi(X),\phi(v))$, along with a choice of preimages $X$ and $v$, specifies a {\it flag directed tree}, meaning each flag (half edge) is directed, but the two directions on an edge needn't be compatible.  If the flag directions are compatible for each edge, then the tree is directed, and so there exists a unique $f^\prime\in Hom_\fcdi(X,v)$ with $\phi(f^\prime)=f$.  If the flag directions are not compatible then such an $f^\prime$ does not exist.  Thus $\phi\colon Hom_\fcdi(X,v)\to Hom_\fccyc(\phi(X),\phi(v))$ is injective and so $\phi$ is faithful.  The factorization axioms follow from the fact that the decomposition of a flag directed tree is flag directed.  Right quadratic preserving follows from the description of $\beta\colon F_dr \Rightarrow R F_c$ as `forgetting directions', meaning edge directions.  The kernel of $\beta$ is generated in weight $1$, by the difference of ways to direct a one edged tree.  Left quadratic preserving follows from the fact that any flag directed tree which is not directed can be assembled starting with a non-directed edge.  Exactness can be seen by exhibiting $L$ as a colimit over a category which is a disjoint union of filtered categories, by contracting those edges in a flag directed tree which are compatibly directed.
\end{proof}

\begin{corollary}
	If $\op{P}$ is a Koszul dioperad, then $L(\op{P})$ and $L^!(\op{P})$ are Koszul cyclic operads.
\end{corollary}

The natural transformation $\sigma$ (see Proposition $\ref{quad}$) can be described as follows.  For an $S^+$-module $E$, define 
\begin{equation*}
\sigma_{r,s}\colon F_c(E)((r+s)) \to F_d(rE)(r,s)
\end{equation*}
by taking an $E$-labeled tree $E(T)$ to the direct sum of all possible directed trees $(T,\succ)$ which in turn yields labels in $rE$.  That is
\begin{equation*}
E(T)\mapsto \ds\bigoplus_{\succ\in \text{Dir}_{r,s}(T)} rE(T, \succ) \subset F_d(rE)(r,s)
\end{equation*} 
where $\text{Dir}_{r,s}(T)$ is the set of all directed trees whose underlying tree is $T$.

To be completely precise, we may apply the above discussion to two variants of dioperads: those which allow empty outputs or empty inputs (but not both) or those which allow neither empty inputs nor empty outputs.  To proceed we will need to differentiate between these variants, so let us denote the Feynman category encoding the former by $\fcdi_0$ and the Feynman category denoting the latter by $\fcdi_+$.  Inclusion $\iota\colon\fcdi_+\hookrightarrow \fcdi_0$  is an admissible morphism of Feynman categories and we may consider the sequence
\begin{equation*}
\fcdi_+ \stackrel{\iota}\hookrightarrow \fcdi_0 \stackrel{\phi_0}\longrightarrow \mathbb{C}
\end{equation*}
whose composite we denote by $\phi_+$.  In particular, $\phi_0$ and $\phi_+$ forget directions, as above.

From the description of $\sigma$ above we see that if we restrict attention to dioperads without empty inputs or outputs then, for the cyclic operad $\op{L}ie$, the image of $\sigma_{2,2}$ of the Jacobi identity is the Drinfeld compatibility criterion for Lie bialgebras.
Moreover, if we let $\op{B}iLie$ be the dioperad encoding Lie bialgebras and $\op{B}al$ be the dioperad encoding balanced infinitesimal bialgebras (see \cite{Agu}) then:  

\begin{lemma}  $R^!_{\phi_+}(\op{L}ie)=\op{B}iLie$ and $R^!_{\phi_+}(\op{A}s)= \op{B}al$.
\end{lemma}

Note that by the functoriality of $R^!$ we immediately recover the result of \cite{Agu} that the commutator and cocommutator in a balanced infinitesimal bialgebra satisfy the compatibility requirement of a Lie bialgebra.

We now turn our attention to algebras over cyclic operads and dioperads.  Recall that if $V$ in $dgVect_k$ is of finite type and $\langle-,-\rangle$ is a symmetric non-degenerate bilinear form, then $End_{(V,\langle-,-\rangle)}$ can be equipped with the structure of a cyclic operad.

For $A$ in $dgVect_k$, again of finite type, we let $A^\ast$ denote its linear dual and we define the double of $A$ to be $ \mathsf{d}(A):= A\oplus A^\ast$.  We equip $\mathsf{d}(A)$ with a bilinear form by defining $\langle a\oplus \eta, b\oplus \xi \rangle= \eta(b)+ \xi(a)$.  This form is easily seen to be symmetric and non-degenerate.

\begin{proposition}\label{bin} (Binomial theorem for cyclic operads.)  There is an isomorphism of cyclic operads:
	\begin{equation*}
	L^!_{\phi_0}(End_A^{ di})\cong End^{ cyc}_{(\mathsf{d}(A),\langle-,-\rangle)}
	\end{equation*}
\end{proposition}
\begin{proof}  If we consider $End_A^{di}(r,s)\cong A^{\tensor r}\tensor (A^{\ast})^{\tensor s}$ then the dioperad structure is given by evaluation.   Here we allow our dioperads to have empty inputs or empty outputs.  The underlying $\mathbb{S}_+$-modules are easily seen to be isomorphic by the binomial theorem.  The fact that these isomorphisms respect the cyclic operad structure on the respective sides can be seen graphically.  In particular, using the form we can identify $End_{\mathsf{d}(A)}^{cyc}((n))$ with $\mathsf{d}(A)^{\tensor n}$.  Hence a pure tensor can be described as a corolla with $n$ flags, each labeled by both an element of $A$ and an element of $A^\ast$.  Applying the binomial theorem splits this corolla into a sum over ways to choose one or the other label.  Composing two flags labeled by  $a\oplus \eta$ and $b\oplus \xi$ before splitting uses the form $\langle-,-\rangle$ and hence yields a factor of $\eta(b)+ \xi(a)$.  Composing the same two legs after splitting gives four sets of terms corresponding to the possible choices of labeling.  The cyclic operad structure of $L^!_{\phi_0}(End_A^{di})$ tells us that the composition of terms labeled by $a$ and $b$ is $0$, as is the composition of terms labeled by $\eta$ and $\xi$, but the other compositions are given by evaluation and hence we once again find a factor of $\eta(b)+ \xi(a)$.
\end{proof}

\begin{corollary}  Let $\op{O}$ be a cyclic operad.  An $R^!_{\phi_0}(\op{O})$-algebra structure on $A$ is equivalent to an $\op{O}$-algebra structure on $(\mathsf{d}(A), \langle-,-\rangle)$.
\end{corollary}

Recall that a Manin triple is equivalent to a morphism of cyclic operads $\op{L}ie \to End^{cyc}_{\mathsf{d}(A)}$ which maps the generator in $\op{L}ie(2)$ to the $(2,1)$ and $(1,2)$ indicies of the direct sum $End^{cyc}_{\mathsf{d}(A)}((3))$. 

Since $L^!_\iota$ is manifest as an extension by zero, combining Proposition $\ref{comp}$ with Proposition $\ref{bin}$ tells us that this is the same thing as a morphism $\op{L}ie\to L^!_{\phi_+}(End_A)$.  Adjointness tells us that such a morphism is in turn equivalent to an algebra over the dioperad $R^!_{\phi_+}(\op{L}ie)=\op{B}iLie$.  We thus see that the adjunction $(R^!,L^!)$ generalizes the equivalence between Manin triples and Lie bialgebras as described in \cite{Drinfeld}.

\subsection{Operads to Cyclic operads (and non-$\Sigma$ variant).}

There is a morphism of Feynman categories $\phi\colon\fcops \to\fccyc$ from the Feynman category for operads to the Feynman category for cyclic operads.  On objects it is given by the identity and morphisms by inclusion.  This morphism restricts to a functor between the Feynman categories for non-$\Sigma$ operads and non-$\Sigma$ cyclic operads.

\begin{lemma}
	The morphism $\phi\colon \fcops\to\fccyc$ is admissible and co-proper.
\end{lemma}
\begin{proof}  The functor $\phi$ may be written as a composition of faithful functors $\fcops\hookrightarrow \fcdi\to\fccyc$, the former being inclusion and the latter having been discussed above in section $\ref{fdex}$.  Hence $\phi$ is faithful.  The factorization axioms and co-properness follow from the fact that a non-rooted tree, assembled via graftings $\xycirc{i}{j}$, can also be assembled via $\circ_i=\xycirc{i}{0}$ by first precomposing with cyclic permutations.
\end{proof}

We thus have a triple of adjoint functors $(L,R,L^!)$.  The usual adjunction $(L,R)$ is reasonably well understood, and we would now like to describe $L^!$.

To do this, we construct a non-$\Sigma$ cyclic operad $\meta$.  First we define $\meta(n)= k[C_{n+1}]$, where $C_{n+1}$ is the cyclic group of order $n+1$, presented as $\{\tau_n: \tau_n^{n+1}=1  \}$, so in particular the subscript always denotes the operad arity.  We then give $\meta$ the structure of a non-$\Sigma$ operad by defining:

\begin{equation*}
\tau_n^r\circ_i\tau_m^s=  \begin{cases} 
	\tau^r & \text{ if } r< i \text{ and } s=0 \\
		\tau^{r+m-1} & \text{ if } r>i \text{ and } s=0 \\
	\tau^{r+s-1} & \text{ if } r=i \text{ and } s\neq 0 \\
	0 & \text{ if }  r\neq i \text{ and } s\neq 0 \text{ or if } (r,s)=(i,0) \\
	 \end{cases} 
\end{equation*}
where we have suppressed the subscript $n+m-1$ on the right hand side.
\begin{lemma}
	As defined above, $\meta$ is a non-$\Sigma$ cyclic operad.
\end{lemma}
\begin{proof}  Let $\tau^r_n$ be depicted as a corolla with $n+1$ flags labeled $\{0\cdc n\}$ whose $r^{th}$ flag is distinguished, say by drawing it in a different color such as green.  Then, the operadic composition can be described as follows: we view $\circ_i$ as gluing flag $0$ on the right to flag $i$ on the left.  If green meets green or black meets black the composition is zero.  Else green meets black, and the composition stipulates that the unmet green flag becomes the new green flag.  This diagrammatic description is easily seen to be associative, from which the claim follows.  Moreover it tells us that $1_1+\tau_1$ is a unit for the composition.\end{proof}

To characterize $L^!$, we may write ``$L^!(\op{P})=\op{P}\tensor\meta$'', however we emphasize that these factors live in different categories.  To give a more precise statement we first recall that a (non-$\Sigma$) cyclic operad is completely characterized by its underlying (non-$\Sigma$) operad and its level-wise cyclic structure.  Then:

\begin{lemma}  Let $\op{P}$ be a (non-$\Sigma$) operad.  Then $L^!(\op{P})$ is the (non-$\Sigma$) cyclic operad whose underlying (non-$\Sigma$) operad is $\op{P}\tensor R(\meta)$ and whose cyclic structure is that induced by $\meta$; explicitly $\tau(p\tensor \tau^r)=p\tensor\tau^{r+1}$.
\end{lemma}

Notice that the cyclic operad $\meta$ arises as the linear dual of the linearization of a cyclic co-operad structure on the groups $C_{n+1}$.  If $\op{P}$ is an operad then the convolution operad $Conv(C_{\ast+1},\op{P})$ is a linear cyclic operad with
$Conv(C_{\ast+1},\op{P})\cong L^!(\op{P})$.  Informally, we view $L^!$ as the operadic analogue of crossing with $S^1$, or of extending from a (co)simplicial set to a cyclic set.  For a concrete realization of this heuristic, consider a multiplicative (non-$\Sigma$) operad $\op{A}s\to \op{P}$.  The collection $\{Conv(C_{n+1},\op{P}(n))\}$ is a cyclic object in $dgVect_k$ lifting the standard cosimplicial structure associated to $\op{A}s\to \op{P}$.

This fact has the following implication in the study of deformation complexes of morphisms of operads and cyclic operads.  An operadic formulation of Deligne's conjecture states that the deformation complex associated to a morphism $\op{A}_\infty\to\op{P}$ carries the structure of a homotopy Gerstenhaber algebra extending the operadic Lie bracket.   An operadic formulation of the cyclic Deligne conjecture states that if the morphism $\op{A}_\infty\to\op{P}$ lifts to the category of non-$\Sigma$ cyclic operads then the homotopy Gerstenhaber structure above lifts to a compatible homotopy $\op{BV}$ algebra.  An $S^1$-equivariant variant of these results states that the deformation complex of this morphism in the category of cyclic operads carries the structure of a $\op{G}rav_\infty$ algebra.  See \cite{Ward2} for full details.

Let us write $\mathsf{Def}(\mu)$ for the deformation complex associated to the map of non-$\Sigma$ operads $\mu\colon\op{A}_\infty\to\op{P}$ and 
$\mathsf{Def}(\tilde{\mu})$ for the deformation complex associated to the adjoint map of non-$\Sigma$ cyclic operads $\tilde{\mu}\colon\op{A}_\infty\to Conv(C_{\ast+1}, \op{P})$. We thus conclude:

\begin{proposition}  
The $\op{G}rav_\infty$ structure on $\mathsf{Def}(\tilde{\mu})$ lifts to a $\op{G}er_\infty$ structure.  The $\op{G}er_\infty$ structure on $\mathsf{Def}(R(\tilde{\mu}))$ lifts to a $\op{BV}_\infty$ structure.  The inclusion $\mathsf{Def}(\tilde{\mu})\hookrightarrow \mathsf{Def}(R(\tilde{\mu})) $ is a $\op{G}er_\infty$ map.
\end{proposition}

\begin{proof}  Note that here we consider $\op{A}_\infty$ to be a non-$\Sigma$ cyclic operad so we may write $\mu\colon R(\op{A}_\infty)\to\op{P}$.  From co-properness $R\cong R^!$, and we can take the adjoint of this morphism to get a map of cyclic operads $\op{A}_\infty\to L^!(\op{P})$.  Using Corollary $\ref{prescor}$ above we may identify $\mathsf{Def}(\tilde{\mu})$ with $\mathsf{Def}(\mu)$ from which the first statement follows.  The other statements then follow immediately from \cite{Ward2} Theorems A and B.
\end{proof}

\subsection{\bf Non-$\Sigma$ operads to operads.}  Let $
\mathbb{PO}$ be the Feynman category encoding non-$\Sigma$ (aka planar) operads and consider the morphism of Feynman categories $\phi\colon\mathbb{PO}\to\mathbb{O}$ which forgets planar structures.
\begin{lemma} $\phi$ is admissible and proper.
\end{lemma}
\begin{proof}
This follows from the fact that every rooted, leaf labeled tree is within a relabeling (permutation) of being embeddable in the plane, in the planar order.  
\end{proof}

Hence $(R^!,L,R)$ is a triple of adjoint functors.  The functors $L$ and $R$ are well understood, e.g.\ $L= -\tensor \op{A}s$.  Theorem $\ref{itertwiningthm}$ recovers the well known fact that $LD\cong DL$.  Note $\phi$ is not right quadratic preserving so $R^!$ is not as simple to describe in general.  However one can show $R^!(\op{L}ie)=\op{A}s$ directly via adjointness or that $R^!(\op{L}_\infty)= \op{A}_\infty$ and $R^!(\op{P}reLie_\infty)= \op{D}end_\infty$ after Theorem $\ref{rint}$.  For example, this recovers the known fact that there is a bijective correspondence between symmetric operad maps $\op{L}ie\to L(\op{P})$ and non-$\Sigma$ operad maps $\op{A}s\to \op{P}$.  We may also compute $R^!(\op{C}om)= \op{N}il$, the nilpotent (non-$\Sigma$) operad (in the parlance of \cite{LV}), having $\op{N}il(2)=k$ and all higher arities are $0$.  In particular this tells us maps $\op{C}om\to L(\op{P})$ for a non-$\Sigma$ operad $\op{P}$ are equivalent to nilpotent elements in $\op{P}(2)$.

A similar analysis can be applied to the morphism associating non-$\Sigma$ cyclic operads and symmetric cyclic operads.

\subsection{Planar dioperads}  There is a Feynman category encoding a planar analog of dioperads.  Its objects are generated by planar graphs with no edges and directed flags.  Its morphisms are generated by planar directed graphs of genus $0$.  Let us call this Feynman category $\mathbb{PDO}$ for planar dioperad and let $\mathbb{PC}$ be the Feynman category encoding non-$\Sigma$ (aka planar) cyclic operads.  There is a morphism $\phi\colon \mathbb{PDO}\to \mathbb{PC}$ by forgetting directed structures.  This morphism may be seen to be admissible via the planar analog of Lemma $\ref{addc}$.

  The dualizing object in the category of planar dioperads is given by $R^!_\phi(\op{A}_\infty)$, where $\op{A}_\infty$ denotes the cyclic $\op{A}_\infty$ operad.

\begin{lemma}  $R^!_\phi(\op{A}_\infty)$ is the planar dioperad encoding $V_\infty$-algebras in the language of \cite{TZ}.
\end{lemma}
\begin{proof}  The intertwining theorems tell us that $R^!_\phi(\op{A}_\infty)= D(I)$ where $I$ is the planar dioperad having the ground field associated to each vertex.  Thus we may describe an algebra over $D(I)$ as having a multi-linear operation corresponding to each planar vertex and subject to the differential condition that the signed sum of one edged planar directed trees evaluates to $0$, with signs coming from the odd sign conventions as in \cite{KWZ}.  This coincides exactly with the description given in Definition 3.1 of \cite{TZ}.  To be precise, we recover their definition provided we adopt the convention of \cite{TZ} that planar directed graphs have non-empty inputs but allow empty outputs in the case that there is at least two inputs.
\end{proof}

Recall that the shifted homology of the moduli space of punctured Riemann spheres forms an operad called the gravity operad which has a Koszul resolution by the dg operad $\op{G}rav_\infty$.

\begin{corollary}  The Lie bracket on the deformation complex of a $V_\infty$-algebra extends to an action of $\op{G}rav_\infty$.
\end{corollary}
\begin{proof}  By Corollary $\ref{prescor}$, the deformation complex of a $V_\infty$-algebra is preserved by the functor $L^!$.  Both of these Feynman categories have symmetrizations given by proper maps, so we may apply Subsection $\ref{fdex}$  to conclude that the deformation complex of $V$ coincides with the deformation complex of its double $V\oplus V^\ast$ viewed as an algebra over a cyclic operad.  It was shown in \cite{Ward2} that any such deformation complex of planar cyclic operads carries an action of such a chain model for the gravity operad, which was shown to be formal in \cite{CWard} hence the claim. \end{proof}

In \cite{TZ} the authors consider actions of string diagrams on cyclic complexes of $V_\infty$ algebras and propose (see their Remark 5.6) that these should be related to moduli spaces of Riemann surfaces.  The above result shows this to be the case in genus $0$.  Let us also observe that this corollary holds for any deformation complex of planar dioperads; the target need not be the endomorphism object.

\subsection{Additional examples.}  By way of conclusion, let us mention several additional examples.

\begin{itemize}
	  
\item {\bf  Shuffle operads to operads}.  The map is given by inclusion.  This map is co-proper; hence $R$ preserves freeness, quadratic objects and the Feynman transform.  Here we can use Manin products on binary subcategories (as pointed out in \cite{LV} 8.10.16) and $R(Lie)$ will be a dualizing object.
	
\item {\bf Half-props to dioperads.}  This map is given by inclusion but is not co-proper.  It was studied in \cite{MV} and we recover several of their results in this framework.
	
\item{\bf Dihedral operads.}  The proper morphism from non-$\Sigma$ cyclic operads to cyclic operads admits a factorization into a pair of proper morphisms via the intermediary of dihedral operads (c.f.\ \cite{AP},\cite{DV}).  In particular, symmetric cyclic operads and dihedral operads are related by a triple of adjoint functors $(R^!,L,R)$ and the study of the dihedral operad $DR(Grav)$ in loc.\ cit.\ can be recast as the study of $R^!(HyCom_\infty)$.
	
\item {\bf Changing colors.}  We can consider Feynman categories encoding colored versions of the above structures; eg colored operads.  Maps between color sets will induce admissible functors on the respective Feynman categories.

\item {\bf Non-$\Sigma$ modular operads to modular operads.}  Non-$\Sigma$ modular operads were introduced in \cite{Marklns}, which gives the adjunction $(L,R)$.  It would be interesting to consider the other operations; for example show that this morphism is proper and to consider $R^!$ and the associated dualizing objects.



\end{itemize}

\appendix

\section{Feynman categories reference guide.}\label{FCA}

In this appendix we will fix our assumptions and conventions regarding Feynman categories.  In particular we choose to impose several restrictions upon the general framework presented in \cite{KW} which will simplify our arguments without limiting the applicability of our results.

Recall that in general a Feynman category is specified by three pieces of data:  a groupoid $\mathbb{V}$, a symmetric monoidal category $(\mathbb{F},\tensor)$ and a functor $\iota\colon\mathbb{V}\to\mathbb{F}$.  This data is then required to satisfy axioms which we recall below.  In this paper we consider all Feynman categories to be strict (Definition 1.8.1 of \cite{KW}) which means that $\iota$ is an inclusion $\mathbb{V}\subset\mathbb{F}$ and we further consider all Feynman categories to be ``skeletal'' by which we mean that the groupoid $\mathbb{V}$ contains only automorphisms.  Let us also impose the condition that $\mathbb{V}$ is a small category and that each automorphism group in $\mathbb{V}$ is finite.
 
After imposing these restrictions, a Feynman category is equivalent to the following data:
\begin{enumerate}
	\item  A set $ob(\mathbb{V})$.  The elements of this set are often called vertices.
	\item  A finite group $Aut(v)=Hom_{\mathbb{V}}(v,v)$ for each vertex $v$.
	\item  A set of ``generating morphisms'' $Hom_{\mathbb{F}}(X,v)$ for each $X\in ob(\mathbb{V})^{\times n}$, each $n\in \mathbb{N}$ and each $v\in ob(\mathbb{V})$.
	\item  A composition law endowing the sets $\{ Hom_{\mathbb{F}}(X,v) \}$ with the structure of a $ob(\mathbb{V})$-colored operad. 
\end{enumerate}

The observation codified in \cite{KW} is that such data is equivalent to the structure of a symmetric monoidal category $(\mathbb{F},\tensor)$ with satisfies the ``hereditary condition'' with respect to $\mathbb{V}$ (see Definition 1.1.1 of \cite{KW}).  Moreover, under this correspondence, algebras over colored operads correspond to symmetric monoidal functors from Feynman categories.  

So, under our assumptions, and up to factors of the monoidal unit, a Feynman category is a pair $(\mathbb{V},\mathbb{F})$ as above, in which the objects of $\mathbb{F}$ can be written as $v_1\tdt v_n$ and the morphisms in $\mathbb{F}$ can be written as $\tensor$-products of compositions of generating morphisms $v_1\tdt v_n\to v_0$ with automorphisms of vertices.  We often denote a Feynman category by $\mathbb{F}$, leaving $\mathbb{V}$ implicit.  

For any fixed symmetric monoidal category $\op{C}$ we define $\mathbb{F}$-$\op{O}ps_\op{C}$ to be the category of strong symmetric monoidal functors from $\mathbb{F}\to\op{C}$ and $\mathbb{V}$-Mods$_\op{C}$ to be the category of functors from $\mathbb{V}$ to $\op{C}$.  In this paper we will consider $\op{C}=dgVect_k$, and omit the subscript $\op{C}$.  Any $\op{O}\in\mathbb{F}$-$\op{O}ps$ is determined by its underlying $\mathbb{V}$-module and its image on generating morphisms. 

Given Feynman categories $\mathbb{F}_1$ and $\mathbb{F}_2$, a morphism between them $\phi\colon\mathbb{F}_1\to\mathbb{F}_2$ is a symmetric monoidal functor which restricts to a functor $\mathbb{V}_1\to\mathbb{V}_2$, which we also call $\phi$.  In this paper we impose that morphisms satisfy the property that for any $Y\in\mathbb{F}_2$, the set $\{X\in \mathbb{F}_1: \phi(X)=Y \}$ is finite.

\subsection{Cubicality}  The conditions imposed above specifying which Feynman categories and morphisms between them that we consider were rather mild.  For some of our results, we need the notion of ``Feynman transform'' and this requires imposing a rather more restrictive condition which we call cubicality (Definition 7.2.2 of \cite{KW}).  Cubicality is an answer to the question: which colored operads have algebras which behave like operads?

A cubical Feynman category is a Feynman category along with a degree function for its morphisms; i.e.\ a function $Mor(\mathbb{F})\to\mathbb{N}$ which is additive with respect to composition and $\tensor$-products.  Cubicality then further requires that the degree function must satisfy two important properties.  First, we know that isomorphisms have degree $0$, but cubicality also requires the converse: a degree $0$ morphism is an isomorphism.  Second, cubicality requires that every degree $n$ morphism (for $n\geq 1$) can be written as a composition of $n$ degree $1$ morphisms is exactly $n!$ distinct ways (up to isomorphism).

To be more precise,  if we define $C^n(A,B)$ to be the set of sequences of $n$ composible degree $1$ morphisms in $\mathbb{F}$ from $A$ to $B$, modulo composition of isomorphisms, then we require $C^n(A,B)$ to have a free and transitive $S_n$ action such that composition of a sequence induces an isomorphism $Hom_\mathbb{F}(A,B)\cong \coprod_n C^n(A,B)_{S_n}$ for which concatenation of sequences corresponds to composition in $\mathbb{F}$.  We refer to Definition 7.2.2 of \cite{KW} for the full details of this definition and further discussion.

The cubicality conditions are often satisfied for Feynman categories whose morphisms are graphs.  The degree is often the number of (internal) edges and the fact that we can assemble $n$ edges in any of the possible $n!$ orders to produce the same graph shows us that the central axiom holds.  Examples include Feynman categories encoding operads, cyclic operads, modular operads, non-$\Sigma$ versions of these, as well as wheeled properads, dioperads, 1/2-props, wheeled operads, dihedral operads, etc.  We can also consider degrees to correspond to multi-edge contractions in which case one could also consider properads.  The fact that all degree $0$ morphisms are isomorphisms forces us to consider all of these objects without units.  This is a condition familiar in the usual bar construction for an operad say, when we first pass to the augmentation ideal of the unit.  It also forces us to restrict our attention to connected graphs.

Finally we define a morphism of cubical Feynman categories $\phi\colon\mathbb{F}_1\to\mathbb{F}_2$ to be a morphism of Feynman categories which is degree preserving.  In particular, to say that $\phi$ is cubical means that both its source and target are cubical, and that $\phi$ is degree preserving.

In the six operations formalism, Verier duality is an endofunctor.  For us the Feynman transform will play this duality role, but having an endofunctor requires a finiteness restriction that we now impose.  Recall (Definition 7.3.1 \cite{KW}) that a cubical Feynman category is of finite type if the number of isomorphism classes of degree $1$ morphisms with a fixed target is finite.  In this paper we impose:

\begin{assumption}\label{finite}  All cubical Feynman categories are of finite type.
\end{assumption}

\subsection{Parity and the Feynman transform.}  Informally, the Feynman transform of an object $\op{O}\in\mathbb{F}$-$\op{O}ps$ is formed by taking the free $\mathbb{F}$-operad on the underlying $\mathbb{V}$-module and imposing a differential formed by summing over degree $1$ morphisms.  In order for this operation to be square zero however, we must introduce some signs.  These signs are encoded by considering the sets $C^n(A,B)$ above with both their given $S_n$ action and the associated alternating action.

One way to make this precise is to introduce enriched Feynman categories and their functors, in which the set $C^n(A,B)$ is replaced with a graded $S_n$ module over $k$.  In particular, to a cubical Feynman category $\mathbb{F}$ we can associate two enriched Feynman categories $\mathbb{F}^+$ and $\mathbb{F}^-$ (see \cite{KW} Definition 5.2.4).  These have the same objects as $\mathbb{F}$, but the morphisms are respectively the trivial enrichment for $\mathbb{F}^+$ or the corresponding alternating action in degree $n$ for $\mathbb{F}^-$.  We can then define the Feynman transform (Definition 7.4.1 of \cite{KW}) as a pair of functors $D^+$ and $D^-$:
\begin{equation*}
D^+\colon \mathbb{F}^+\text{-}\op{O}ps\leftrightarrows\mathbb{F}^-\text{-}\op{O}ps :D^-
\end{equation*}

These functors are of the form $D^\pm(-)= (F^\mp(G(-)^\ast),\partial)$.  Theorem 7.4.3 of \cite{KW} then states that $D^+D^-\sim id^-$ and $D^-D^+\sim id^+$.

Another way to introduce the categories $\mathbb{F}^\pm$-$\op{O}ps$ is to first consider the Feynman category $\hat{\mathbb{F}}$ whose objects are the same as $\mathbb{F}$ and whose morphisms are $Hom_{\hat{\mathbb{F}}}(A,B)\cong \coprod_nC^n(A,B)$ and then to consider $\mathbb{F}^\pm$-$\op{O}ps\subset\hat{\mathbb{F}}$-$\op{O}ps$ to be those functors which are (skew) invariant: $\op{P}(\sigma(f))=(\pm 1)^{|\sigma|}\op{P}(f)$ for $\sigma\in S_{|f|}$.

In this paper we will typically suppress the superscript $\pm$ notation.  There are two equivalent ways in which this can be interpreted.  First, this may be considered to mean that a statement is valid for both choices of superscripts.  Of course, in the presence of multiple instances of $D$, the choices must be consistent.  For example, the statement $D^2\sim id$ means that both $D^+D^-\sim id^-$ and $D^-D^+\sim id^+$ (which is true).  Alternatively, we can consider the category $\mathbb{F}^+\text{-}\op{O}ps\coprod \mathbb{F}^-\text{-}\op{O}ps$ whose set of objects is the disjoint union of two objects sets and with no additional morphisms.  In this case we can consider $D:=D^+\coprod D^-$.
Notice $\mathbb{F}^+\text{-}\op{O}ps\coprod \mathbb{F}^-\text{-}\op{O}ps$ is a symmetric monoidal category in which $\tensor$ respects parity in the expected way.

There is a third way to interpret the suppression of superscripts in certain cases.  Often there exists an isomorphism  $\mathbb{F}^+\text{-}\op{O}ps \cong \mathbb{F}^-\text{-}\op{O}ps$ under which $D^+\cong D^-$.  In such a case we can view $D^+\cong D^-$ as an endomorphism.  In practice, these isomorphisms are realized by various shifts and suspensions, see Table $IV$ of \cite{KWZ}.  This is the case for operads and, up to this isomorphism, $D$ is the usual bar construction.  This is not the case for modular operads, for example, and here our category of $\mathbb{F}^-\text{-}\op{O}ps$ corresponds to the $\mathfrak{K}$-modular operads of \cite{GeK2}. 

\subsection{Quadratic objects and quadratic duality.}

One consequence of $\mathbb{F}$ being cubical is that the free objects in $\mathbb{F}$-$\op{O}ps$ are $\mathbb{N}$-graded by the degree of morphisms.  We call this grading the ``weight'' and depict it as a superscript.  Notice that for  $E\in\mathbb{V}$-Mods, the weight $r$ component of the free $\mathbb{F}$-operad can be viewed as a $\mathbb{V}$-module; i.e.\ $F(E)^r\in \mathbb{V}$-Mods.

Quadratic data relative to a given cubical Feynman category $\mathbb{F}$ is a pair $(E,S)$, where $E\in\mathbb{V}$-Mods which is finite dimensional in each arity, and $S\subset F(E)^1$ as a $\mathbb{V}$-module.  Let us denote the free $\mathbb{F}$-operad as $F=F_+$ and the free $\mathbb{F}_-$-operad as $F_-$.  These spaces are still weight graded, and in weight 1 we have  $F(E)^1 = F_-(E)^1$.  Thus quadratic data also determines a subspace $S\subset F_-(E)^1$.

We define the category $\mathbb{F}_\op{Q}$-$\op{O}ps$ as follows.  An object is a pair of a quadratic datum and an element of $\{+,-\}=\mathbb{Z}_2$, called the parity.  The morphisms always preserve parity and are given by equivariant maps $E\to E^\prime$ such that the induced map $F_\pm(E)^1\to F_\pm (E^\prime)^1$ sends $S$ in to $S^\prime$.  Quadratic duality is a functor $\mathbb{F}_\op{Q}\text{-}\op{O}ps \to \mathbb{F}_\op{Q}\text{-}\op{O}ps$ defined by $(E,S,\pm) \mapsto (E^\ast, S^\perp, \mp)$, where $\ast$ denotes linear dual and $S^\perp$ denotes those functionals vanishing on $S$.  And $\pm$ means (of course) $+$ or $-$, whence $\mp$ means we have changed parity.  Clearly this functor is involutive.

We let $\langle S\rangle $ denote the ideal generated by $S$.  This is the sub-$\mathbb{V}$-module of $F_\pm(E)$ generated by all images of $F_{\pm}(E)(-)(-)$ evaluated at morphisms $f\in Mor(\mathbb{F})$ and vectors in $F_\pm(E)(s(f))$ having at least one tensor factor in $S$.  It is an easy exercise to see that for $S\subset F_\pm(E)$, the quotient $F_\pm(E)/\langle S \rangle$ is naturally in $\mathbb{F}_\pm$-$\op{O}ps$.

There is a faithful functor $\mathbb{F}_\op{Q}$-$\op{O}ps \to \mathbb{F}$-$\op{O}ps$ given by sending $(E,S,\pm) \mapsto F_\pm(E)/\langle S \rangle$.  Here (as above) we abuse notation by considering $\mathbb{F}$-$\op{O}ps=\mathbb{F}^+\text{-}\op{O}ps\coprod \mathbb{F}^-\text{-}\op{O}ps$.    The objects in the image of this functor are called quadratic $\mathbb{F}$-$\op{O}ps$, and we often (further) abuse notation by writing $F_\pm(E)/\langle S \rangle\in \mathbb{F}_\op{Q}$-$\op{O}ps$.  If $\op{O}$ is quadratic in $\mathbb{F}$-$\op{O}ps$, then we let $\op{O}^!$ denote the quadratic object associated to the dual quadratic data.  Note, as above, in the presence of the usual shift and suspension isomorphisms $\mathbb{F}^+\text{-}\op{O}ps \cong \mathbb{F}^-\text{-}\op{O}ps$, and this definition recovers the usual notion of quadratic duality (after composing with said isomorphism).  For example, our notion of quadratic dual of the commutative operad is the ``odd Lie operad'': $\Sigma^{-1}\fr{s}\op{L}ie$, see \cite{KWZ}.

Finally, we observe that quadratic duality and the Feynman transform are related by a natural morphism $D(\op{P})\to\op{P}^!$.  For if $\op{P}=F_\pm(E)/\langle S \rangle$ then there is an inclusion of $\mathbb{V}$-modules $E\hookrightarrow GP$, and hence a sequence $F_\mp(GP^\ast)\to F_\mp(E^\ast)\to \op{P}^!$.  The composition in this sequence is non-zero only for vertices labeled by generators, and hence for cycles in $D(\op{P})$.  Thus composition of this sequence gives a dg map $D(\op{P})\to\op{P}^!$.  We say $\op{P}\in\mathbb{F}_\op{Q}$-$\op{O}ps$ is Koszul if this map is a weak equivalence.

{\small 
\subsection*{Acknowledgment:}
I would like to gratefully acknowledge the support of the Simons Center for Geometry and Physics where this paper was conceived of.  In particular lectures delivered there by John Morgan entitled ``A topologist looks at sheaf theory'' as part of the program ``Interactions of Homotopy Theory and Algebraic Topology with Physics through Algebra and Geometry'' helped form the conceptual basis for this work.  I have also benefited from helpful discussions with Alexander Berglund, Gabriel Drummond-Cole, Martin Markl, Dan Petersen, Dennis Sullivan, Bruno Vallette and Sasha Voronov.  I am grateful for the support of Stockholm University and IHES where portions of this paper were written.  This paper was improved significantly by detailed and helpful suggestions and comments from an anonymous referee.  Finally, this paper owes an immeasurable debt to my advisor Ralph Kaufmann who freely shared his knowledge and ideas with me over the course of countless conversations.
}


\begin{thebibliography}{CHKV16}
	
	\bibitem[Agu01]{Agu}
	Marcelo Aguiar.
	\newblock On the associative analog of {L}ie bialgebras.
	\newblock {\em J. Algebra}, 244(2):492--532, 2001.
	
	\bibitem[AP17]{AP}
	J~Alm and D~Petersen.
	\newblock Brown's dihedral moduli space and freedom of the gravity operad.
	\newblock {\em Ann. Sci. \'Ec. Norm. Sup\'er. (4)}, 50(5):1081--1122, 2017.
	
	\bibitem[Bar07]{Bar}
	Serguei Barannikov.
	\newblock Modular operads and {B}atalin-{V}ilkovisky geometry.
	\newblock {\em Int. Math. Res. Not. IMRN}, (19):Art. ID rnm075, 31, 2007.
	
	\bibitem[BK]{KB}
	C.~Berger and R.~Kaufmann.
	\newblock Derived decorated {F}eynman categories and moduli spaces.
	\newblock {\em in preparation}.
	
	\bibitem[BM08]{BM}
	Dennis~V. Borisov and Yuri~I. Manin.
	\newblock Generalized operads and their inner cohomomorphisms.
	\newblock In {\em Geometry and dynamics of groups and spaces}, volume 265 of
	{\em Progr. Math.}, pages 247--308. Birkh\"auser, Basel, 2008.
	
	\bibitem[Bor94]{FB}
	Francis Borceux.
	\newblock {\em Handbook of categorical algebra. 2}, volume~51 of {\em
		Encyclopedia of Mathematics and its Applications}.
	\newblock Cambridge University Press, Cambridge, 1994.
	\newblock Categories and structures.
	
	\bibitem[CGP16]{CGP}
	Melody Chan, Soren Galatius, and Sam Payne.
	\newblock The tropicalization of the moduli space of curves ii: Topology and
	applications.
	\newblock {\em arxiv.org/abs/1604.03176}, 2016.
	
	\bibitem[CHKV16]{CHKV}
	James Conant, Allen Hatcher, Martin Kassabov, and Karen Vogtmann.
	\newblock Assembling homology classes in automorphism groups of free groups.
	\newblock {\em Comment. Math. Helv.}, 91(4):751--806, 2016.
	
	\bibitem[CKV13]{CKV}
	James Conant, Martin Kassabov, and Karen Vogtmann.
	\newblock Hairy graphs and the unstable homology of {${\rm Mod}(g,s)$}, {${\rm
			Out}(F_n)$} and {${\rm Aut}(F_n)$}.
	\newblock {\em J. Topol.}, 6(1):119--153, 2013.
	
	\bibitem[CV03]{CV}
	James Conant and Karen Vogtmann.
	\newblock On a theorem of {K}ontsevich.
	\newblock {\em Algebr. Geom. Topol.}, 3:1167--1224, 2003.
	
	\bibitem[CW18]{CWard}
	Ricardo Campos and Benjamin~C. Ward.
	\newblock Gravity formality.
	\newblock {\em Adv. Math.}, 331:439--483, 2018.
	
	\bibitem[Dri87]{Drinfeld}
	V.~G. Drinfeld.
	\newblock Quantum groups.
	\newblock {\em Proceedings of the {I}nternational {C}ongress of
		{M}athematicians, {V}ol. 1, 2 ({B}erkeley, {C}alif., 1986)}, pages 798--820,
	1987.
	
	\bibitem[DV17]{DV}
	Cl\'ement Dupont and Bruno Vallette.
	\newblock Brown's moduli spaces of curves and the gravity operad.
	\newblock {\em Geom. Topol.}, 21(5):2811--2850, 2017.
	
	\bibitem[FHM03]{FHM}
	H.~Fausk, P.~Hu, and J.~P. May.
	\newblock Isomorphisms between left and right adjoints.
	\newblock {\em Theory Appl. Categ.}, 11:No. 4, 107--131, 2003.
	
	\bibitem[Gan03]{Gan}
	Wee~Liang Gan.
	\newblock Koszul duality for dioperads.
	\newblock {\em Math. Res. Lett.}, 10(1):109--124, 2003.
	
	\bibitem[Gia11]{Gian}
	Jeffrey Giansiracusa.
	\newblock The framed little 2-discs operad and diffeomorphisms of handlebodies.
	\newblock {\em J. Topol.}, 4(4):919--941, 2011.
	
	\bibitem[GK94]{GiK}
	Victor Ginzburg and Mikhail Kapranov.
	\newblock Koszul duality for operads.
	\newblock {\em Duke Math. J.}, 76(1):203--272, 1994.
	
	\bibitem[GK95]{GeK1}
	E.~Getzler and M.~M. Kapranov.
	\newblock Cyclic operads and cyclic homology.
	\newblock In {\em Geometry, topology, \& physics}, Conf. Proc. Lecture Notes
	Geom. Topology, IV, pages 167--201. Int. Press, Cambridge, MA, 1995.
	
	\bibitem[GK98]{GeK2}
	E.~Getzler and M.~M. Kapranov.
	\newblock Modular operads.
	\newblock {\em Compositio Math.}, 110(1):65--126, 1998.
	
	\bibitem[Kau14]{KSurvey}
	Ralph~M. Kaufmann.
	\newblock Arc geometry and algebra: foliations, moduli spaces, string topology
	and field theory.
	\newblock In {\em Handbook of {T}eichm\"uller theory. {V}ol. {IV}}, volume~19
	of {\em IRMA Lect. Math. Theor. Phys.}, pages 255--325. Eur. Math. Soc.,
	Z\"urich, 2014.
	
	\bibitem[KLP03]{KLP}
	Ralph~M. Kaufmann, Muriel Livernet, and R.~C. Penner.
	\newblock Arc operads and arc algebras.
	\newblock {\em Geom. Topol.}, 7:511--568 (electronic), 2003.
	
	\bibitem[KM01]{KM}
	M.~Kapranov and Yu. Manin.
	\newblock Modules and {M}orita theorem for operads.
	\newblock {\em Amer. J. Math.}, 123(5):811--838, 2001.
	
	\bibitem[Kon93]{Kont2}
	Maxim Kontsevich.
	\newblock Formal (non)commutative symplectic geometry.
	\newblock In {\em The {G}elfand {M}athematical {S}eminars, 1990--1992}, pages
	173--187. Birkh\"auser Boston, Boston, MA, 1993.
	
	\bibitem[Kon94]{Kont}
	Maxim Kontsevich.
	\newblock Feynman diagrams and low-dimensional topology.
	\newblock In {\em First {E}uropean {C}ongress of {M}athematics, {V}ol.\ {II}
		({P}aris, 1992)}, volume 120 of {\em Progr. Math.}, pages 97--121.
	Birkh\"auser, Basel, 1994.
	
	\bibitem[KW17]{KW}
	Ralph~M. Kaufmann and Benjamin~C. Ward.
	\newblock Feynman categories.
	\newblock {\em Ast\'erisque}, (387):vii+161, 2017.
	
	\bibitem[KWZ15]{KWZ}
	Ralph~M. Kaufmann, Benjamin~C. Ward, and J.~Javier Z{\'u}{\~n}iga.
	\newblock The odd origin of {G}erstenhaber brackets, {B}atalin-{V}ilkovisky
	operators, and master equations.
	\newblock {\em J. Math. Phys.}, 56(10):103504, 40, 2015.
	
	\bibitem[LV08]{LVor}
	A.~Lazarev and A.~A. Voronov.
	\newblock Graph homology: {K}oszul and {V}erdier duality.
	\newblock {\em Adv. Math.}, 218(6):1878--1894, 2008.
	
	\bibitem[LV12]{LV}
	Jean-Louis Loday and Bruno Vallette.
	\newblock {\em Algebraic operads}, volume 346 of {\em Grundlehren der
		Mathematischen Wissenschaften}.
	\newblock Springer, Heidelberg, 2012.
	
	\bibitem[Man87]{Manin}
	Yu.~I. Manin.
	\newblock Some remarks on {K}oszul algebras and quantum groups.
	\newblock {\em Ann. Inst. Fourier (Grenoble)}, 37(4):191--205, 1987.
	
	\bibitem[Mar16]{Marklns}
	Martin Markl.
	\newblock Modular envelopes, {OSFT} and nonsymmetric (non-{$\Sigma$}) modular
	operads.
	\newblock {\em J. Noncommut. Geom.}, 10(2):775--809, 2016.
	
	\bibitem[MV09]{MV}
	M.~Markl and A.~A. Voronov.
	\newblock P{ROP}ped-up graph cohomology.
	\newblock In {\em Algebra, arithmetic, and geometry: in honor of {Y}u. {I}.
		{M}anin. {V}ol. {II}}, volume 270 of {\em Progr. Math.}, pages 249--281.
	Birkh\"auser Boston, Inc., Boston, MA, 2009.
	
	\bibitem[TZ07]{TZ}
	Thomas Tradler and Mahmoud Zeinalian.
	\newblock Algebraic string operations.
	\newblock {\em $K$-Theory}, 38(1):59--82, 2007.
	
	\bibitem[War16]{Ward2}
	Benjamin~C. Ward.
	\newblock Maurer-{C}artan elements and cyclic operads.
	\newblock {\em J. Noncommut. Geom.}, 10(4):1403--1464, 2016.
	
	\bibitem[War18]{Wardsdp}
	Benjamin~C. Ward.
	\newblock Intertwining for semi-direct product operads.
	\newblock {\em to appear in Algebr. Geom. Topol.}, 2018.
	
\end{thebibliography}
\end{document}